\documentclass[reqno,12 pt]{amsart}%amsart

\usepackage{amsfonts,amscd,mathrsfs}
\usepackage{amsthm}
\usepackage{amssymb}
\usepackage{amsmath}
\usepackage{verbatim}
\usepackage{mathrsfs}
\usepackage{latexsym}
\usepackage{enumerate}
\usepackage{graphicx}
\usepackage[width=7cm]{caption}
\usepackage{color}
\usepackage{epsf}
\usepackage{geometry}
\usepackage{url}
\usepackage{hyperref}
\definecolor{dark-blue}{rgb}{0,0,0.6} 	
\definecolor{Purple}{rgb}{0.2,0,0.25}
\hypersetup{breaklinks=true,
pagecolor=white,
colorlinks=true,%false
citecolor=dark-blue,%{rgb}{0.2,0.2,1},%
filecolor=black,%
linkcolor=Purple,%
urlcolor=black
}
\newcommand{\bref}[1]{\textbf{\ref{#1}}} %bold font for any cross reference 
\newcommand{\beqref}[1]{\textbf{(\ref{#1})}} %bold font for any equation number

\theoremstyle{plain} 
\newtheorem{thm}{Theorem}[section]
\newtheorem{lem}[thm]{Lemma}
\newtheorem{defin}[thm]{Definition}
\newtheorem{cor}[thm]{Corollary}
\newtheorem{prop}[thm]{Proposition}

\theoremstyle{definition}
\newtheorem{remark}[thm]{Remark}
\newtheorem{expl}[thm]{Example}
\numberwithin{equation}{section}

\newcommand{\wt}{\widetilde}

\newcommand{\abs}{\textnormal{abs}}
\newcommand{\conv}{\textnormal{conv}}

\newcommand{\R}{\mathbb{R}}

\newcommand{\N}{\mathbb{N}}

\newcommand{\M}{\mathscr{M}}

\newcommand{\Kbo}{\mathscr{K}_{\textnormal{bound},(0)}}


\begin{document}
\date{April 8, 2019}

\title[Fixed points of polarity type operators]{Fixed points of polarity type operators}

\author{Daniel Reem}
\address{Daniel Reem, Department of Mathematics, The Technion -- Israel Institute of Technology, 3200003 Haifa, Israel.} 
\email{dream@technion.ac.il}
\author{Simeon Reich}
\address{Simeon Reich, Department of Mathematics, The Technion -- Israel Institute of Technology, 3200003 Haifa, Israel.} 
\email{sreich@technion.ac.il}
\maketitle

\begin{abstract}
A well-known result says that the Euclidean unit ball is the unique fixed point of the polarity operator. This result implies that if, in $\R^n$, the unit ball of some norm is equal to the unit ball of the dual norm, then the norm must be Euclidean. Motivated by these results and by relatively recent results in convex analysis and convex geometry regarding various properties of order reversing operators, we consider, in a real Hilbert space setting, a more general fixed point equation in which the polarity operator is composed with a continuous invertible linear operator. We show that if the linear operator is positive definite, then the considered equation is uniquely solvable by an ellipsoid. Otherwise, the equation can have several (possibly infinitely many) solutions or no solution at all.  Our analysis yields a few by-products of possible independent interest, among them results related to coercive bilinear forms (essentially a quantitative convex analytic converse to the celebrated Lax-Milgram theorem from partial differential equations) and a characterization of real Hilbertian spaces. 
\end{abstract}

\section{Introduction}
\subsection{Background:}
Consider the following geometric fixed point equation: 
\begin{equation}\label{eq:FixedPointConvexGeometry}
C=(GC)^{\circ}.
\end{equation}
Here $C\neq\emptyset$ is the unknown subset which is assumed to be contained in a given real Hilbert space $X\neq\{0\}$, $G:X\to X$ is a given continuous, invertible and linear operator, $GC:=\{Gc: c\in C\}$, and $S^{\circ}$ denotes the polar (or dual) of $\emptyset\neq S\subseteq X$ (see \beqref{eq:C^polar} below). 

In this paper we analyze and solve \beqref{eq:FixedPointConvexGeometry} under various assumptions on $C$ and on $G$. The motivation to consider \beqref{eq:FixedPointConvexGeometry} is based on a number of reasons. First, \beqref{eq:FixedPointConvexGeometry} is a generalization of the equation 
\begin{equation}\label{eq:C=Cpolar}
C=C^{\circ},
\end{equation}
which describes all the self-polar sets. A well-known and classical result in convex geometry says that there exists a unique self-polar set and this set is the unit ball (see, for example, \cite[p. 138]{Barvinok2002book}, \cite[pp. 144--145]{BauschkeCombettes2017book}, \cite[p. 148]{Hiriart-UrrutyLemarechal2001book}). This result implies that if we start with $\R^n$ and want to define on it a norm such that the unit ball induced by this norm coincides with the unit ball of the dual norm, then we can do this if and only if the norm is Euclidean (here both balls are considered  subsets of $\R^n$ and we identify $\R^n$ with its dual space).

A second reason for considering \beqref{eq:FixedPointConvexGeometry} originates in a relatively recent branch of research in convex geometry. In some of the works belonging to this branch, certain order reversing operators (such as  isomorphisms, involutions, or operators satisfying certain functional equations involving sets) acting on various classes of finite-dimensional geometric objects were considered. A central property that was established there was that these operators must have the form $T(C)=LC^{\circ}$, where $L$ is some invertible  linear operator. For instance, in \cite[Corollary, p. 659]{BoroczkySchneider2008jour} the objects are compact and convex subsets of $\R^n$ containing the origin in their interior, in \cite[Corollary 1]{Schneider2008jour} the objects are closed and convex cones, in \cite[Theorem 10]{Artstein-AvidanMilman2008},  \cite[Corollary 1.14]{MilmanSegalSlomka2011jour} and \cite[Corollary 5]{Slomka2011jour} the objects are closed and convex subsets of $\R^n$ containing the origin, and in \cite[Corollary 1.11]{Artstein-AvidanSlomka2012jour} the objects are $n$-dimensional centrally symmetric ellipsoids (in all of these works $n\in\N$ satisfies either $n\geq 2$ or $n\geq 3$; see also \cite[Theorem 4]{SegalSlomka2013chapt} for a closely related but somewhat different result based on a characterization involving fractional linear mappings). Equation \beqref{eq:FixedPointConvexGeometry} is directly related to these works because, as a simple verification shows (see Lemma \bref{lem:gamma-C-is-one-to-one}\beqref{item:gamma_C composition} below), it can be written as 
\begin{equation}\label{eq:C=LCpolar}
C=LC^{\circ},
\end{equation}
where $L=(G^*)^{-1}$. In other words, the operator on the right-hand side of \beqref{eq:FixedPointConvexGeometry}, namely the one which maps $C$ to $(GC)^{\circ}$ (we consider this operator to be a ``polarity type'' operator), can be written as $T(C)=LC^{\circ}$, as in the works mentioned above.  Hence our work can be thought of as being a continuation of the above-mentioned branch of research in the ``fixed point direction''. 

A third reason for considering \beqref{eq:FixedPointConvexGeometry} is the following fixed point equation which has recently been introduced in \cite[Equation (1.1)]{IusemReemReich2019jour}: 
\begin{equation}\label{eq:f_Tf}
f(x)=\tau f^*(Ex+c)+\langle w,x\rangle +\beta,\quad x\in X. 
\end{equation}
Here $X$ is a real Hilbert space, $f:X\to[-\infty,\infty]$ is 
the unknown function, $\tau>0$, $c\in X$, $w\in X$, $\beta\in\R$ are given, and $E:X\to X$ is a 
given continuous linear invertible operator. In addition, 
\begin{equation}\label{eq:f^*}
f^*(x^*):=\sup\{\langle x^*,x\rangle-f(x): x\in X\}, \quad x^*\in X, 
\end{equation}
is the Legendre-Fenchel transform (namely, the convex conjugate) of $f$. Equation \beqref{eq:f_Tf} can be thought of as being a convex analytic version of \beqref{eq:FixedPointConvexGeometry} not only because of some similarities in their structure, but also because of several similarities in the properties of the corresponding solution sets. For example, in both cases the solution sets are very sensitive to the various parameters which appear there (see Subsection \bref{subsec:Contributions} and Theorem \bref{thm:FixedPointConvexGeometry} below regarding  \beqref{eq:FixedPointConvexGeometry}). In addition, as we show in Lemma \bref{lem:f=f^*G*} and  Sections \bref{sec:X=R}--\bref{sec:NonExistence}, there is a strong relation between some of the results mentioned in \cite{IusemReemReich2019jour} (for instance, Lemma 5.1, Lemma 5.2, Lemma 7.1, Proposition 9.1, Example 13.2) and some of the results of our paper.  

We note that if in \beqref{eq:f^*} one restricts attention to lower semicontinuous convex and proper functions, then the operator $T$ which maps each $f$ to the right-hand side of \beqref{eq:f_Tf} is the most general fully order reversing operator which acts on this class of functions (namely, it is invertible, both the operator and its inverse reverse the point-wise order between functions, and any other order reversing operator which acts on the class of lower semicontinuous convex and proper functions $f:X\to(-\infty,\infty]$ must have the form $T$): this is shown in  \cite[Theorem 7]{ArtsteinMilman2009} (finite-dimensional spaces) and \cite[Theorem 2]{IusemReemSvaiter2015jour} (arbitrary infinite-dimensional Banach spaces; here a few modifications are needed regarding the various parameters and variables  which appear on the right-hand side of \beqref{eq:f_Tf}, among them that $Tf$ is defined on the dual $X^*$ of $X$ and is lower  semicontinuous in the weak$^*$ topology, and that $\langle x^*,x\rangle:=x^*(x)$ for each $x\in X$ and $x^*\in X^*$). It is worth noting that several years ago other convex analytic versions of \beqref{eq:f_Tf} were discussed: an equation which characterizes self-polar functions \cite{Rotem2012jour} and versions related to generalized self-conjugate functions \cite{AlvesSvaiter2011, Svaiter2003}.

Finally, we note that as far as we know, there has been no systematic attempt to investigate \beqref{eq:FixedPointConvexGeometry} so far. However, one can see, in a few cases which are scattered in the literature, that particular cases of \beqref{eq:FixedPointConvexGeometry} have been considered briefly,  mainly in a different terminology. For instance, we have already mentioned places where \beqref{eq:C=Cpolar} has been discussed. In addition, in \cite[pp. 122, 130]{BauschkeCombettes2017book}, \cite[p. 140]{Cegielski2012book}, \cite[pp. 23, 28]{DraganMorozanStoica2010book}, and \cite[p. 392]{Goffin1980jour} it is said briefly that some cones solve \beqref{eq:FixedPointConvexGeometry} in the special case where $G=-I$ (see Example \bref{ex:Cone} below for more details regarding this latter claim; in \cite{BauschkeCombettes2017book, DraganMorozanStoica2010book} these cones are called ``self-dual''). Furthermore, in \cite[p. 147]{Barvinok2002book} a few examples are given of polytopes in a finite-dimensional Euclidean space (namely, compact polyhedra having 0 in their interiors) which satisfy \beqref{eq:FixedPointConvexGeometry} for some invertible linear operators $G$ (such polytopes are called ``self-dual'' \cite[p. 147]{Barvinok2002book}). The point of view there is, however, different from the point of view of our paper, since in \cite[p. 147]{Barvinok2002book} one starts with some closed and convex subset $C$ and tries to find an invertible linear operator $G$ such that \beqref{eq:C=LCpolar} (hence \beqref{eq:FixedPointConvexGeometry}) will hold. In other words, the unknown in \cite[p. 147]{Barvinok2002book} is $G$ and not $C$. 

\subsection{Contributions:}\label{subsec:Contributions}
The main result of this paper is Theorem \bref{thm:FixedPointConvexGeometry} below which analyzes \beqref{eq:FixedPointConvexGeometry} and describes its set of solutions under some assumptions on the linear operator $G$ and on the class of sets in which we seek the unknown $C$. This theorem shows that \beqref{eq:FixedPointConvexGeometry} can have no solution, a unique solution or several (possibly infinitely many) solutions. More precisely, the theorem states the following: 
\begin{thm}\label{thm:FixedPointConvexGeometry}
Let $(X,\langle\cdot,\cdot\rangle)$ be a real Hilbert space (satisfying $X\neq\{0\}$) and let $G:X\to X$ be a continuous and invertible linear operator. Consider equation \beqref{eq:FixedPointConvexGeometry} with an unknown $\emptyset\neq C\subseteq X$. The following statements hold:
\begin{enumerate}[(i)]
\item\label{item:ClosedConvex} Any solution to \beqref{eq:FixedPointConvexGeometry} must be closed and convex,  and must contain 0. 
\item\label{item:PositiveDefinite} If $G$ is positive definite, then there exists a unique solution to \beqref{eq:FixedPointConvexGeometry} and this solution is an ellipsoid having the form $C=\{x\in X: \langle Gx,x\rangle\leq 1\}$. 
\item\label{item:NotPositiveDefinite} If $G$ is not positive definite, then there are cases where \beqref{eq:FixedPointConvexGeometry} has several (possibly infinitely many) solutions  and cases where \beqref{eq:FixedPointConvexGeometry} does not have any solution which belongs to the class of  bounded subsets of $X$ that contain 0 in their interiors. 
\end{enumerate}
\end{thm}
Our analysis, which is somewhat different from analyses that are frequently used in fixed point theory \cite{GranasDugondji2003book, KirkSims2001book}, yields a few by-products of possibly independent interest, among them results related to coercive bilinear forms (essentially a quantitative converse of the celebrated Lax-Milgram theorem from partial differential equations: see Lemma \bref{lem:EllipticOperator}, Remark \bref{rem:Coercive}, and Lemma \bref{lem:EllipticOperatorGeneraized}), a characterization of real Hilbertian  spaces (namely Banach spaces which are isomorphic to Hilbert spaces: see Remark \bref{rem:LaxMilgram}) and  results related to infinite-dimensional convex geometry (for instance, Lemma \bref{lem:(gamma_C)_polar=gamma_(C_polar)} and Lemma \bref{lem:Ellipsoid}). We also note that although our analysis is essentially dimension-free (in the sense that no quantitative expressions involving the dimension appear; the only exception is Example \bref{ex:Simplex} below), in some cases the dimension does appear in ``the back door'': for instance, in Proposition \bref{prop:X=R} we consider the case where $X=\R$ and classify completely the set of solutions to \beqref{eq:FixedPointConvexGeometry}, and in Examples \bref{ex:Ellipsoid}--\bref{ex:L1Linfty} we present a few two-dimensional examples of non-uniqueness. 

The intuition behind Theorem \bref{thm:FixedPointConvexGeometry} is mainly inspired by  analogous results mentioned in \cite{IusemReemReich2019jour}, especially \cite[Theorem 3.1]{IusemReemReich2019jour}, but one can get some intuition also by thinking of particular cases, such as the ones presented in Figures \bref{fig:SquareRhombus}--\bref{fig:Orthant}. The  proof of Theorem \bref{thm:FixedPointConvexGeometry} can be described briefly as follows: Part \beqref{item:ClosedConvex} is just an  immediate consequence of \beqref{eq:FixedPointConvexGeometry} and the definitions; Part \beqref{item:PositiveDefinite} is based on certain properties of ellipsoids, mainly the ones described in Lemma \bref{lem:Ellipsoid}; the proof of Part \beqref{item:NotPositiveDefinite} is by separating into cases: the existence part is again related to properties of ellipsoids and also to a certain operator equation (see \beqref{eq:AG}) which allows one to construct an explicit solution to  \beqref{eq:FixedPointConvexGeometry} based on some properties of the operator $G$; the proof of the non-uniqueness case is itself done by separating into subcases: the one-dimensional subcase is completely analyzed directly, and in higher dimensions the analysis is based on either another  operator equation (see \beqref{eq:MoreSolutions}) or a direct analysis related to specific examples which are closely related to Figures \bref{fig:SquareRhombus}--\bref{fig:Orthant} below; finally, the proof of the non-existence case is based on introducing a special type of operators called ``semi-skew operators'', choosing $G$ to be such an operator, and using several relations between polar sets and functions, as well as several connections between \beqref{eq:FixedPointConvexGeometry} and \cite{IusemReemReich2019jour} and a general convex analytic lemma (Lemma \bref{lem:UpperBound}), to obtain a contradiction if a solution to \beqref{eq:FixedPointConvexGeometry} is assumed (in the class of  bounded subsets of $X$ that contain 0 in their interiors). 

\subsection{Paper layout:} After some preliminaries which are given in Section \bref{sec:Preliminaries}, we present in Section \bref{sec:Tools} several auxiliary results needed in the proof of Theorem \bref{thm:FixedPointConvexGeometry}. The proof is developed in Sections \bref{sec:GisPositiveDefinite}--\bref{sec:NonExistence} and is presented formally in Section \bref{sec:Proof}. Section \bref{sec:ConcludingRemarks} concludes the paper with a few remarks and open problems. We end the paper with an appendix (Section \bref{sec:Appendix}) which contains the proofs of various assertions mentioned in the text without proof.

\section{Preliminaries}\label{sec:Preliminaries}
Throughout the paper we assume that $X$ is a real Hilbert space satisfying $X\neq\{0\}$ and endowed with an inner product $\langle\cdot,\cdot\rangle$. The induced norm is denoted by $\|\cdot\|$. We say that $f:X\to [-\infty,\infty]$ is proper whenever $f(x)>-\infty$ for all $x\in X$ and, in addition, $f(x)\neq\infty$ for at least one point $x\in X$. 
The convex conjugate (Legendre-Fenchel transform) of $f$ is the function 
$f^*:X\to [-\infty,\infty]$ which is defined in \beqref{eq:f^*}. 

Given a linear and continuous operator
$G:X\to X$, its adjoint is the linear operator $G^*:X\to X$ defined by 
the equation $\langle G^*a,b\rangle =\langle a,Gb\rangle$ for all $(a,b)\in
X^2$. It is well known that $G^*$ exists, is unique, and is continuous (well-known results mentioned here without a proof can be found, for instance, in \cite{Brezis2011book, Kreyszig1978book, Shalit2017book}). The norm of $G$ is $\|G\|:=\sup\{\|Gx\|/ \|x\|: 0\neq x\in X\}$. It is well known that  $G$ is continuous if and only if it is bounded (namely $\|G\|<\infty$), and if $G$ is invertible, then its inverse $G^{-1}$ is continuous. We say that $G$ is self-adjoint (or symmetric) if $G$ is continuous and $G=G^*$. If $G$ is self-adjoint and satisfies 
$\langle Gx,x\rangle\geq 0$ for all $x\in X$, then we say that $G$ is positive semidefinite. A self-adjoint operator $G:X\to X$ satisfying 
$\langle Gx,x\rangle>0$ for all $0\neq x\in X$ is called positive definite. We denote by $I:X\to X$ the identity operator, that is, $I(x):=x$ for each $x\in X$. For a subset $\emptyset\neq C\subseteq X$ we denote $GC:=\{Gc: c\in C\}$. 

We say that $B:X^2\to \R$ is a bilinear form whenever both $x\mapsto B(x,y)$ and $y\mapsto B(x,y)$ are linear functions from $X$ to $\R$ for each $y\in X$ and $x\in X$, respectively. It is a well-known fact that if $B$ is a continuous bilinear form, then there exists a unique continuous linear operator $A:X\to X$ satisfying $B(x,y)=\langle Ax,y\rangle$ for all $x,y\in X$. In this case we say that $A$ generates $B$.  We say that a bilinear form $B$ is coercive (instead of ``coercive'', the terms ``elliptic'', ``strongly coercive'' and ``strongly monotone'' are of use too) if there exists $\beta>0$ (the coercivity coefficient of $B$) such that $B(x,x)\geq \beta\|x\|^2$ for each $x\in X$. 

For a unit vector $u\in X$ we  denote by $u^{\bot}$ the set $\{x\in X: \langle x,u\rangle=0\}$, namely the orthogonal complement of $u$. It is well known  that $X$ is isometric to $\R u\times u^{\bot}$ (endowed with the inner product $\langle (r_1u,v_1),(r_1u,v_2)\rangle:=r_1r_2+\langle v_1,v_2\rangle$ for every $r_1,r_2\in \R$, $v_1,v_2\in u^{\bot}$). The orthogonal projection from $X$ onto a closed linear subspace $M$ of $X$ is denoted by $P_M$. For $x,y\in X$, we denote by $[x,y]$ the closed  line segment which connects $x$ and $y$, namely the set $\{x+t(y-x): t\in [0,1]\}$. We denote by $\Kbo(X)$ the set of all bounded, convex and closed subsets $C$ of $X$ having 0 in their interior. We say that $C\subseteq X$ is a centrally symmetric ellipsoid if it has the form $C=\{x\in X: \langle Gx,x\rangle\leq 1\}$ for some positive definite operator $G:X\to X$. 

Given a nonempty subset $C\subseteq X$, the gauge (or Minkowski functional) associated with $C$ is the function $\M_C:X\to[0,\infty]$ defined by 
\begin{equation}\label{eq:gamma_C}
\M_C(x):=\inf\{\mu\geq 0: x\in\mu C\},\quad x\in X,
\end{equation}
where, of course, $\mu C=\{\mu c: c\in  C\}$ and $\inf\emptyset:=\infty$. If we further assume that $C\in \Kbo(X)$, then it can easily be verified that 
\begin{equation}\label{eq:Inequality gamma_C norm}
\frac{\|x\|}{\|C\|}\leq\M_C(x)\leq \frac{\|x\|}{r_C}, \quad\,\forall\, x\in X,
\end{equation}
where $r_C>0$ is the radius of any open ball which is contained in $C$ and containing the origin, and $\|C\|:=\sup\{\|c\|: c\in C\}\geq r_C$. In particular, in this case $\M_C$ is finite everywhere, $\M_C(0)=0$, and $\M_C(x)>0$ for every $x\neq 0$, and, moreover, $\M_C$ is positively homogenous (namely $\M_C(\lambda x)=\lambda\M_C(x)$ for all $x\in X$ and $\lambda\geq 0$). In addition, $\M_C$ is subadditive \cite[p. 26]{VanTiel1984book}, that is, $\M_C(x+y)\leq \M_C(x)+\M_C(y)$ for all $x,y\in X$. Combining this inequality with the fact that $\M_C$ is positively homogenous, we see that $\M_C$ is convex. Furthermore, the assumption $C\in \Kbo(X)$ also implies that $\M_C$ is Lipschitz continuous because its subadditivity and \beqref{eq:Inequality gamma_C norm} imply the inequality 
\begin{equation}\label{eq:M_CIsLipschitz}
|\M_C(x)-\M_C(y)|\leq\max\{\M_C(x-y),\M_C(y-x)\}\leq \frac{1}{r_C}\|x-y\|,\quad\forall x,y\in X.
\end{equation}

 Given $C\in \Kbo(X)$, the polar of $\M_C$ is the function $\M_C^{\circ}:X\to[0,\infty]$ defined by 
\begin{multline}\label{eq:gamma_C^polar}
\M_C^{\circ}(x^*):=\sup\left\{\frac{\langle x^*,x\rangle}{\M_C(x)}: 0\neq x\in X\right\}\\
=\sup\left\{\langle x^*,x\rangle: x\in X, \M_C(x)=1 \right\},\quad x^*\in X,
\end{multline}
where the right-most expression follows from the fact that $\M_C$ is positively homogenous (and the left-most inequality in \beqref{eq:Inequality gamma_C norm} ensures that $\M_C(x)>0$ for all $0\neq x\in X$). Since $C\in \Kbo(X)$, a simple verification shows that $\M_C^{\circ}$ is finite and positive everywhere with the exception of the origin at which it vanishes. Actually, in this case  
\begin{equation}\label{eq:Inequality gamma_C^polar}
r_C\|x^*\|\leq \M_C^{\circ}(x^*)\leq\|C\|\|x^*\|,\quad \forall x^*\in X,
\end{equation}
where the right-most inequality follows from \beqref{eq:Inequality gamma_C norm} and the Cauchy-Schwarz inequality, and the left-most inequality is a consequence of the well-known fact that $\|x^*\|=\sup\{\langle x^*,x\rangle: x\in X, \|x\|=1\}$ (indeed, this latter equality implies that for each $\alpha\in (0,1)$, there exists $x_{\alpha}\in X$ such that $\|x_{\alpha}\|=1$ and $\langle x^*,x_{\alpha}\rangle\geq \alpha\|x^*\|$; thus from \beqref{eq:gamma_C^polar} and \beqref{eq:Inequality gamma_C norm} we have $\M_C^{\circ}(x^*)\geq \langle x^*,x_{\alpha}\rangle/\M_C(x_{\alpha})\geq \alpha\|x^*\|r_C$ for all $\alpha\in (0,1)$, namely $\M_C^{\circ}(x^*)\geq\|x^*\|r_C$, as claimed).  Moreover, \beqref{eq:gamma_C^polar} implies that $\M_C^{\circ}$ is positively homogenous and subadditive. Thus if $C\in \Kbo(X)$, then also $\M_C^{\circ}$ is convex and Lipschitz continuous.

The polar (or dual) of $\emptyset\neq C\subseteq X$ is the set 
\begin{equation}\label{eq:C^polar}
C^{\circ}:=\{x^*\in X: \langle x^*,c\rangle\leq 1\,\,\forall c\in C\}.
\end{equation}

%%%%%%%%%%%%%%%%%%%%%%%%%%%%%%%%%%%%%%%%%%%%%%%%%%%%%%%%%%%%%%%%%%%%
\begin{figure}[t]
\begin{minipage}[t]{0.45\textwidth}
\begin{center}{\includegraphics[scale=0.57]{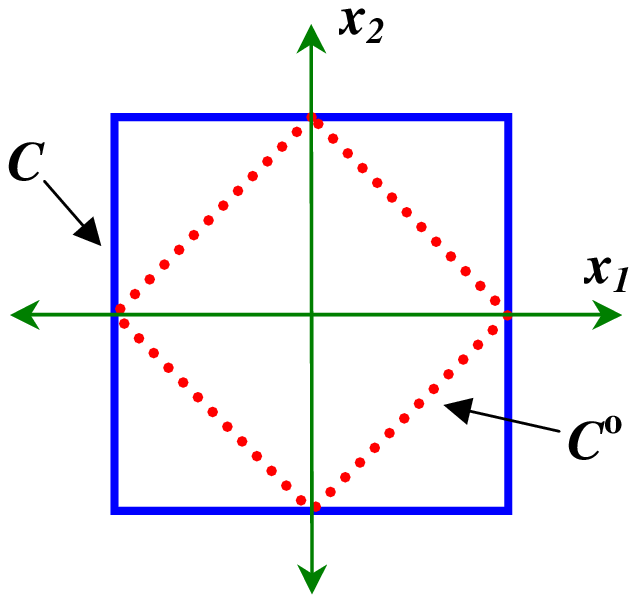}}%{alg40.eps}}
\end{center}
 \caption{$C=\{(x_1,x_2)\in\R^2: \max\{|x_1|,|x_2|\}\leq 1\}$, $C^{\circ}=\{(x_1,x_2)\in\R^2: |x_1|+|x_2|\leq 1\}$ (only boundaries are shown); see also Example \bref{ex:L1Linfty}.}
\label{fig:SquareRhombus}
\end{minipage}
\hfill
\begin{minipage}[t]{0.45\textwidth}
\begin{center}
{\includegraphics[scale=0.63]{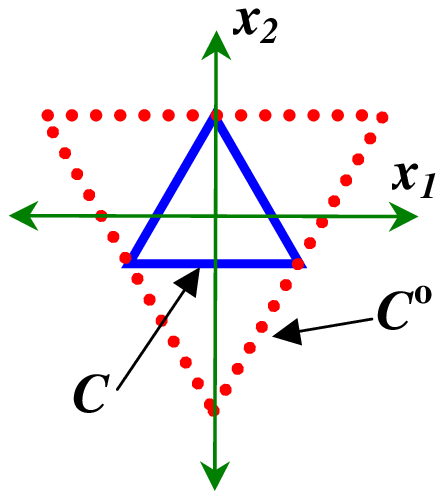}}
\end{center}
 \caption{$C$ is a regular 2D simplex of circumradius 1, $C^{\circ}$ is a regular simplex of circumradius 2 (only boundaries are shown); see also Example \bref{ex:Simplex}.}
\label{fig:Triangle}
\end{minipage}
\vfill
\begin{minipage}[t]{0.45\textwidth}
\begin{center}{\includegraphics[scale=0.65]{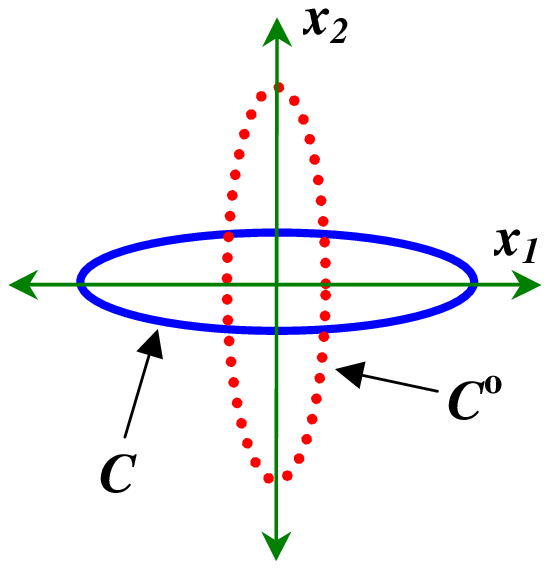}}%{alg40.eps}}
\end{center}
 \caption{$C=\{(x_1,x_2)\in\R^2: (1/4)x_1^2+4x_2^2\leq 1\}$, $C^{\circ}=\{(x_1,x_2)\in\R^2: 4x_1^2+(1/4)x_2^2\leq 1\}$ (only boundaries are shown); see also Lemma \bref{lem:Ellipsoid}, Proposition \bref{prop:PositiveDefiniteG} and Example \bref{ex:Ellipsoid}}
\label{fig:Ellipse}
\end{minipage}
\hfill
\begin{minipage}[t]{0.45\textwidth}
\begin{center}
{\includegraphics[scale=0.65]{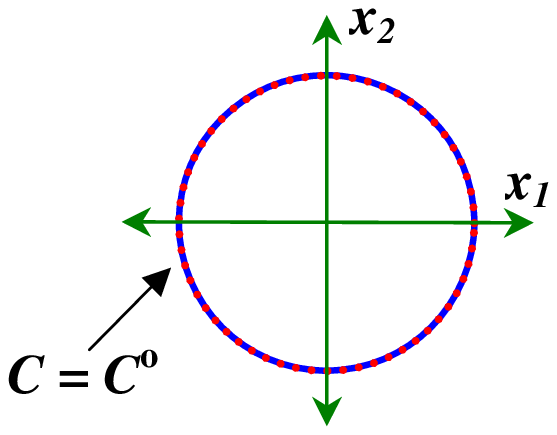}}
\end{center}
 \caption{$C$ is the unit disc, and so is $C^{\circ}$ (only boundaries are shown); see also Lemma \bref{lem:Ellipsoid} and  Proposition \bref{prop:PositiveDefiniteG}.}
\label{fig:Disc}
\end{minipage}
\vfill
\begin{minipage}[t]{1\textwidth}
\begin{center}
{\includegraphics[scale=0.58]{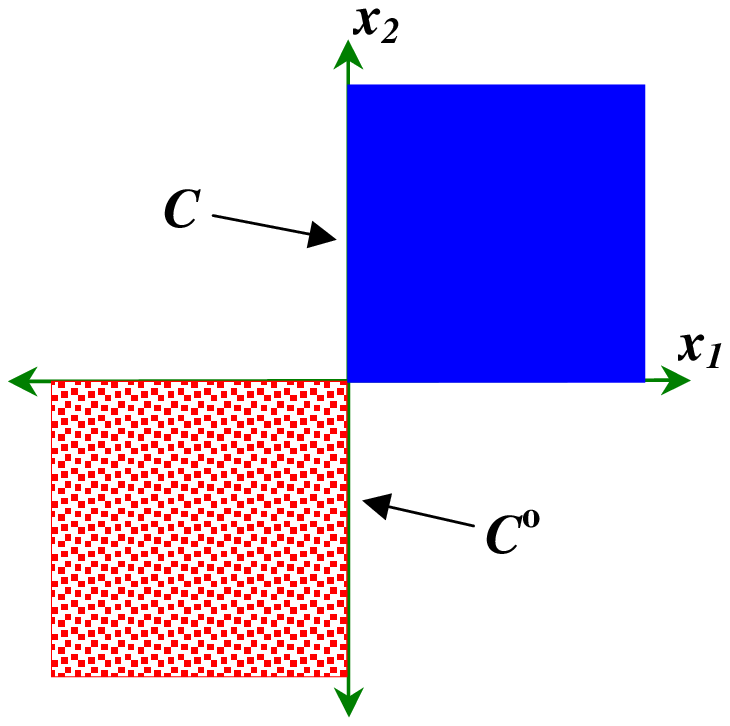}}
\end{center}
 \caption{$C$ is the nonnegative orthant, $C^{\circ}$ is the nonpositive orthant; see also Example  \bref{ex:Cone}.}
\label{fig:Orthant}
\end{minipage}
\end{figure}
%%%%%%%%%%%%%%%%%%%%%%%%%%%%%%%%%%%%%%%%%%%%%%%%%%%%%%%%%%%%%%%%%%%%

It can immediately be verified  that $0\in C^{\circ}$ and that $C^{\circ}$ is closed and convex. Moreover, the map $C\mapsto C^{\circ}$ is order reversing, namely, if $\emptyset\nsubseteq C_1\subseteq C_2\subseteq X$, then $C_2^{\circ}\subseteq C_1^{\circ}$. A few illustrations of sets and their polar sets are given in Figures \bref{fig:SquareRhombus}--\bref{fig:Orthant}. It can be seen, at least intuitively, that for all of the sets $C$ mentioned in these figures it is possible to find an invertible linear operator which transforms $C^{\circ}$ into $C$, and so these sets satisfy \beqref{eq:C=LCpolar} and hence also \beqref{eq:FixedPointConvexGeometry}. 

The above-mentioned concepts are, of course, closely related to norms and dual norms. Indeed, if $C\in \Kbo(X)$ is symmetric with respect to the origin (that is, $C=-C)$, then it is immediate to check that $\M_C(x)=\M_C(-x)$ for all $x\in X$. Hence \beqref{eq:Inequality gamma_C norm} implies that $\M_C$ is a norm which is equivalent to the Hilbertian norm. Moreover, Lemma \bref{lem:gamma-C-is-one-to-one}\beqref{item:C=C(1)} below implies that $C$ is the unit ball associated with $\M_C$. The  definition of the dual norm and \beqref{eq:gamma_C^polar} ensure that the dual norm of $\M_C$ is $\M_C^{\circ}$ (that is, $\M_{C^{\circ}}$, as follows from Lemma \bref{lem:(gamma_C)_polar=gamma_(C_polar)}), and when combined with \beqref{eq:C^polar}, this implies that the dual unit ball of $C$ is $C^{\circ}$. Finally, since $\M_C$ is equivalent to the Hilbertian norm, the dual space of $\M_C$ is the same as $X^*\cong X$.

\section{Auxiliary results}\label{sec:Tools}
In this section we present a few auxiliary results which are used in later sections (Sections \bref{sec:GisPositiveDefinite}--\bref{sec:NonExistence}). Some of these results partly or fully extend, to an infinite-dimensional setting, several well-known results mentioned in \cite[Section 15]{Rockafellar1970book} in one way or another. As far as we know, most of the results below, in particular, the ones related to ellipsoids, including the ideas used in the proofs (whenever the considered results generalize known results), are new in the setting which we consider (an exception is Lemma \bref{lem:gamma-C-is-one-to-one}\beqref{item:C=C(1)} which appears, in a more general formulation, elsewhere, say in \cite[Corollary 14.13, p. 242]{BauschkeCombettes2017book}; however, in the setting that is relevant to us our proof seems to be much simpler). 

Nevertheless, one may wonder what the difference in the proofs between the finite- and infinite-dimensional cases is. Well, it is not always easy to describe this difference. As an illustration, consider Lemma \bref{lem:Ellipsoid}\beqref{item:EllipsoidPolarSet}  in our paper about the polar of a centrally symmetric ellipsoid. In \cite{Rockafellar1970book} one can see the analogous finite-dimensional statement on page 136, and the proof of it is based on \cite[Corollary 15.3.2]{Rockafellar1970book}; however, the proof of this corollary is based on \cite[Corollary 15.3.1]{Rockafellar1970book}, which is based on \cite[Theorem 15.3]{Rockafellar1970book}, which is based on \cite[Theorem 8.6]{Rockafellar1970book}, which is based on \cite[Theorem 8.5]{Rockafellar1970book}, which is based on \cite[Theorem 8.3]{Rockafellar1970book}, which is based on \cite[Theorem 8.2]{Rockafellar1970book}, which is based on \cite[Corollary 6.8.1]{Rockafellar1970book}, which is based on \cite[Theorem 6.8]{Rockafellar1970book}, which is based on \cite[Theorem 6.6]{Rockafellar1970book}, which is based on \cite[Corollary 6.3.1, p. 46]{Rockafellar1970book}, which is finite-dimensional since it says that the closures of two convex subsets $C_1$ and $C_2$ are equal if and only if their relative interiors are equal (an infinite-dimensional counterexample to it can simply be obtained by taking a Hilbert space $C_2$ which contains a dense linear subspace $C_1\neq C_2$; hence $\overline{C_1}=\overline{C_2}=C_2$, but the relative interior of $C_2$ is $C_2$ itself and hence it strictly contains the relative interior of $C_1$). 

In order to be on the safe side, and for the sake of completeness and convenience, we decided to include full proofs of all of the claims that are relevant to our paper; however, because of the nature of these claims and in order to improve the flow of the paper, these proofs are given in an appendix (Section \bref{sec:Appendix}). 

We start with a lemma which describes several properties of the Minkowski functional. 
\begin{lem}\label{lem:gamma-C-is-one-to-one} 
Consider our  real Hilbert space $X$. Then: 
\begin{enumerate}[(a)]
\item\label{item:C=C(1)} Given $C\in \Kbo(X)$, let $C(1):=\{x\in X: \M_{C}(x)\leq 1\}$. Then $C=C(1)$. 
\item For each $x\in X\backslash\{0\}$ and $C\in \Kbo(X)$, we have $(1/\M_C(x))x\in C$.
\item\label{item:gamma_C is one-to-one} For all $C_1, C_2\in \Kbo(X)$, if $\M_{C_1}=\M_{C_2}$, then $C_1=C_2$. 
\item\label{item:gamma_C composition} Suppose that $G:X\to X$ is a continuous and invertible linear operator. Then for an arbitrary $\emptyset\neq C\subseteq X$ one has $\M_{GC}=\M_C\circ G^{-1}$ and $(GC)^{\circ}=(G^*)^{-1}C^{\circ}$. Moreover, if $C\in \Kbo(X)$, then $GC\in \Kbo(X)$ and  $(\M_{GC})^{\circ}=\M_C^{\circ}\circ G^*$. 
\end{enumerate}
\end{lem}

The next lemma presents a certain duality between the polar function and the polar set.  
\begin{lem}\label{lem:(gamma_C)_polar=gamma_(C_polar)}
Let $C\in \Kbo(X)$. Then $C^{\circ}\in \Kbo(X)$ and $\M_C^{\circ}=\M_{C^{\circ}}$. 
\end{lem}

Our next lemma shows a certain relation between conjugacy and polarity. 
\begin{lem}\label{lem:gamma_C_conjugate_polar}
Let $\phi:\R\to(-\infty,\infty]$ and assume that $\phi(t)=\infty$ for every $t\in (-\infty,0)$. Given our real Hilbert space $X$, for each $C\in \Kbo(X)$ and each $x^*\in X$, we have 
\begin{equation}\label{eq:gamma_C_conjugate_polar}
(\phi\circ\M_C)^*(x^*)=\phi^*(\M^{\circ}_C(x^*)). 
\end{equation}
Moreover, if, in addition, we also assume that $\phi$ is finite and differentiable over $[0,\infty)$ (with a right derivative at 0), that $\phi(0)=0$, and that $\phi'$ is strictly increasing on $[0,\infty)$ and maps this interval onto itself, then for all $x^*\in X$ and  $C\in \Kbo(X)$,
\begin{equation}\label{eq:phi(gamma_C)}
\begin{array}{lll}
(\phi\circ \M_C)^*(x^*)&=(\phi')^{-1}(\M_C^{\circ}(x^*))\M_C^{\circ}(x^*)-\phi((\phi')^{-1}(\M_C^{\circ}(x^*)))\\
 & =\displaystyle{\int_0^{\M_C^{\circ}(x^*)}(\phi')^{-1}(t)dt}.
\end{array}
\end{equation}
In particular, for all $x^*\in X$ and $C\in \Kbo(X)$,
\begin{equation}\label{eq:0.5MC^2}
\left(\frac{1}{2}\M_C^2\right)^*(x^*)=\frac{1}{2}(\M_C^{\circ}(x^*))^2.
\end{equation}
\end{lem}

The following lemma (which, despite its simplicity, is new to the best of our knowledge) is needed in the proof of Lemma \bref{lem:Ellipsoid} below, mainly in order to show that ellipsoids induced by invertible positive  definite operators are  bounded subsets.  It is possible to prove Lemma \bref{lem:Ellipsoid} using the existence and uniqueness of a positive semidefinite square root of a given positive semidefinite operator, a fact which follows from either the theory of Banach algebras \cite[Theorem 12.33, p. 331]{Rudin1991book} or by more specific considerations related to bounded linear operators in Hilbert spaces \cite[Theorem 9.4-2, p. 476]{Kreyszig1978book} (some caution is needed here because the standard setting in which the above-mentioned fact is proved is a complex Hilbert space; however, at least in the case of \cite[Theorem 9.4-2, p. 476]{Kreyszig1978book} a very slight modification of the proof is needed so that it also holds in a real Hilbert space setting). Instead of using the above-mentioned fact related to square roots of positive semidefinite operators, we  prefer to present an elementary and  purely convex analytic proof, based on Lemma \bref{lem:QuadConj} below (see Subsection \bref{subsec:ProofsSectionTools} below for the proof). One advantage of our proof is that it can be generalized virtually word for word to a more general setting (as done in Remark \bref{rem:LaxMilgram} and Subsection \bref{subsec:LaxMilgramGeneralized} below), while it is not clear to us how to apply the techniques of \cite[Theorem 12.33, p. 331]{Rudin1991book} or \cite[Theorem 9.4-2, p. 476]{Kreyszig1978book} in that more general setting. 

\begin{lem}\label{lem:EllipticOperator}
Given a real Hilbert space $X$, if $A:X\to X$ is a positive semidefinite and invertible linear operator, then $A$ is coercive (in particular, $A$ is positive definite). As a matter of fact, 
\begin{equation}\label{eq:EllipticOperator}
\langle Ax,x\rangle\geq \|A^{-1}\|^{-1}\|x\|^2,\quad \forall\, x\in X  
\end{equation}
and $\|A^{-1}\|^{-1}$ is the optimal (largest possible) coercivity coefficient.  
\end{lem}
As said above, for proving Lemma \bref{lem:EllipticOperator} we need the following well-known lemma, the proof of which is just a simple calculation (see also Lemma \bref{lem:QuadConjGeneralized} below for a more general statement).
\begin{lem}\label{lem:QuadConj}
Let $X$ be a real Hilbert space and let $A:X\to X$ be a positive semidefinite invertible operator. For each $x\in X$, denote $h(x):=\frac{1}{2}\langle Ax, x\rangle$. Then $h^*(x^*)=\frac{1}{2}\langle A^{-1}x^*,x^*\rangle$ for all $x^*\in X$.
\end{lem}

\begin{remark}\label{rem:Coercive}
Lemma \bref{lem:EllipticOperator} is closely related to the celebrated Lax-Milgram theorem from partial differential equations \cite[Corollary 5.8, p. 140]{Brezis2011book}, \cite[Theorem 2.1, p. 169]{LaxMilgram1954incol}. This theorem essentially says that given a continuous bilinear form $B:X^2\to\R$, if $B$ is coercive, then its generating operator $A$ (that is, the continuous linear operator $A:X\to X$ satisfying $B(x,y)=\langle Ax,y\rangle$ for each $x,y\in X$) is invertible. Even without the Lax-Milgram theorem it is immediate that if $B$ is symmetric (namely $B(x,y)=B(y,x)$ for all $x,y\in X$), then $A=A^*$, and if $B$ is also coercive, then $A$ is positive definite. Hence if we restrict attention to the common case of symmetric, continuous and coercive bilinear forms $B$, then we can conclude from the Lax-Milgram theorem that $A$ is invertible and positive definite. Lemma \bref{lem:EllipticOperator} implies essentially a quantitative converse: if $A$ is positive definite (or merely positive semidefinite) and invertible, then $B$ is coercive, and, moreover, the best possible coercivity coefficient of $B$ is $1/\|A^{-1}\|$. Since coercive bilinear forms have applications in other areas, such as calculus of variations (for example,  Stampacchia's theorem \cite[Theorem 5.6, pp. 138, 145]{Brezis2011book}, \cite[Th\'eor\`eme 1]{Stampacchia1964jour}), Lemma \bref{lem:EllipticOperator} may find applications in these areas too. We also note that Lemma ~\bref{lem:EllipticOperator} generalizes \cite[Remark 15.3]{IusemReemReich2019jour}. 
\end{remark}

The following lemma discusses a few properties of ellipsoids. 
\begin{lem}\label{lem:Ellipsoid}
Given a positive definite and invertible linear operator $A:X\to X$ where $X$ is a real Hilbert space, let $D:=\{x\in X: \langle Ax,x\rangle\leq 1\}$ be the centrally symmetric ellipsoid induced by $A$. The following statements hold:
\begin{enumerate}[(a)]
\item\label{item:EllipsoidIsBounded} $D\in \Kbo(X)$. 
\item\label{item:EllipsoidPolarFunction} $\M_{D}(x)=\sqrt{\langle Ax,x\rangle}$ for each $x\in X$. 
\item\label{item:EllipsoidPolarSet} $D^{\circ}=\{x\in X: \langle A^{-1}x,x\rangle\leq 1\}$.
\end{enumerate}
\end{lem}

 We note that the assumption that $A^{-1}$ exists is crucial for Lemma \bref{lem:Ellipsoid}. For instance, as a counterexample for Part  \beqref{item:EllipsoidIsBounded}, one can to take $X:=\ell_2$ and $A$ to be any diagonal operator the entries of which are $\lambda_k$ so that $\lambda_k>0$ for every $k\in \N$ and $\lim_{k\to\infty}\lambda_k=0$. This $A$ is positive definite but not invertible and each of the vectors $x_k:=(1/\sqrt{\lambda_k})e_k$ belongs to  the centrally symmetric ellipsoid $D:=\{x\in X: \langle Ax,x\rangle\leq 1\}$ induced by $A$, that is, $D$ is not bounded (here $e_k$ is the $k$-th element of the canonical basis of $\ell_2$, namely the $k$-th component of $e_k$ is 1 and the other components are zero).

The next ``convex analytic lemma'' is somewhat known in the sense that versions of it have been mentioned in the literature in more restricted settings (see, for instance, \cite[Example 5, p. 349]{CensorReem2015jour} and \cite[p. 288]{Eppstein2005incol}). 

\begin{lem}\label{lem:UpperBound}
Let $X$ be a real normed space and assume that $f:X\to(-\infty,\infty]$ is convex. If $f$ is bounded above by some $\alpha\in\R$, then $f$ is identically equal  to some real constant.  
\end{lem}

 The following proposition follows immediately from \beqref{eq:FixedPointConvexGeometry} and the fact that the polar of a set is closed, convex and contains the origin. 
\begin{prop}\label{prop:ClosedConvex0}
If $C$ solves \beqref{eq:FixedPointConvexGeometry}, then $C$ is closed, convex and $0\in C$.
\end{prop}
We finish this section by describing a connection between \beqref{eq:FixedPointConvexGeometry} and \beqref{eq:f_Tf}. 
\begin{lem}\label{lem:f=f^*G*}
If $C\in \Kbo(X)$ solves \beqref{eq:FixedPointConvexGeometry}, then $f:=\frac{1}{2}\M_C^2$ satisfies the equation 
\begin{equation}\label{eq:f=f^*G*}
f(x)=f^*(G^* x),\quad x\in X.
\end{equation}
\end{lem}

\section{$G$ is positive definite: existence and uniqueness}\label{sec:GisPositiveDefinite} 
Here we show that \beqref{eq:FixedPointConvexGeometry} has a unique solution when $G$ is positive definite. 

\begin{prop}\label{prop:PositiveDefiniteG}
If $G:X\to X$ is a positive definite and invertible linear operator, then \beqref{eq:FixedPointConvexGeometry} has a unique solution $\emptyset\neq C\subseteq X$. The solution is the ellipsoid $C=\{x\in X: \langle Gx,x\rangle\leq 1\}$.
\end{prop}

\begin{proof}
Let $D:=\{x\in X: \langle Gx,x\rangle\leq 1\}$. We first show that if $C=D$, then $C$ solves \beqref{eq:FixedPointConvexGeometry}. Indeed, since $C=D$ and $GC=\{Gx: x\in X,\,\langle Gx,x\rangle\leq 1\}$, the invertibility of $G$ and the change of variables $y=Gx$ show that $GC=\{y\in X: \,\langle y,G^{-1}y\rangle\leq 1\}$. By using this identity, the fact that the inner product is symmetric and Lemma \bref{lem:Ellipsoid}\beqref{item:EllipsoidPolarSet}, we see that $(GC)^{\circ}=\{x\in X: \,\langle Gx,x\rangle\leq 1\}$, namely $(GC)^{\circ}=C$, as required. 

Now we show that if $C$ solves \beqref{eq:FixedPointConvexGeometry}, then it must be that $C=D$. Let $x\in C$ be arbitrary. From \beqref{eq:FixedPointConvexGeometry} we have $x\in (GC)^{\circ}$. Hence $\langle x,Gc\rangle\leq 1$ for all $c\in C$ and, in particular, for $c=x$. It follows that $x\in D$, namely $C\subseteq D$. Therefore $GC\subseteq GD$, and since the polarity operation reverses the order, we get $(GC)^{\circ}\supseteq (GD)^{\circ}$. Since we assume that $C$ solves \beqref{eq:FixedPointConvexGeometry} and since we already know from the previous paragraph that $D$ solves \beqref{eq:FixedPointConvexGeometry}, it follows that $C\supseteq D$. We conclude that $C=D$, as required. 
\end{proof}

Another existence and uniqueness result, in a somewhat restricted setting, is described in Proposition \bref{prop:BallUniqueEllipsoid} below. 

\section{The one-dimensional case}\label{sec:X=R}
In this section we classify completely the set of solutions to \beqref{eq:FixedPointConvexGeometry} when $X=\R$. We note that since $G$ is linear and invertible, its form must be $G(x)=\gamma x$ for every $x\in X$, where $\gamma$ is a fixed positive or negative real number. 
\begin{prop}\label{prop:X=R}
 Suppose that $X=\R$ and that $G:X\to X$ is linear and invertible. Then the following statements hold:
\begin{enumerate}[(a)]
\item If $G$ has the form $G(x)=\gamma x$ for each $x\in X$, where $\gamma>0$ is fixed, then $C=\{x\in \R: |x|\leq 1/\sqrt{\gamma}\}$ is the unique solution to \beqref{eq:FixedPointConvexGeometry}.
\item If $G$ has the form $G(x)=\gamma x$ for each $x\in X$, where $\gamma<0$ is fixed, then the set of solutions to \beqref{eq:FixedPointConvexGeometry} consists of the following sets: 
\begin{enumerate}[(i)]
\item The sets $C_b:=[\frac{1}{\gamma b},b]$, where $b>0$ is arbitrary (namely, infinitely many  sets which belong to $\Kbo(X)$), and  
\item two unbounded sets which contain 0 on their boundaries, namely the rays $C_{-}:=(-\infty,0]$ and $C_{+}:=[0,\infty)$. 
\end{enumerate}
\end{enumerate}
\end{prop}

\begin{proof}
If $\gamma>0$, then $G$ is positive definite. Hence Proposition \bref{prop:PositiveDefiniteG} ensures that the unique solution to   \beqref{eq:FixedPointConvexGeometry} is $C=\{x\in \R: G(x)\cdot x\leq 1\}$, namely $C=\{x\in \R: |x|\leq 1/\sqrt{\gamma}\}$, as claimed. 

From now one we assume that $\gamma<0$. Proposition \bref{prop:ClosedConvex0} ensures that any solution $C$ to \beqref{eq:FixedPointConvexGeometry} is convex, closed and contains 0. Since $X=\R$, this fact implies that $C$ has the form $C=[a,b]$, where $-\infty\leq a\leq 0\leq b\leq \infty$ (here we use the notation $[-\infty,b]:=(-\infty,b]$ where $b\in \R$, and so forth). Hence is it sufficient to  consider the following 9 cases according to the values of the parameters $a$ and $b$, and to check directly in each  case  whether $C$ solves \beqref{eq:FixedPointConvexGeometry}: 
\begin{enumerate}[(I)]
\item $-\infty=a<0=b<\infty$
\item\label{case:C=(-infty,0]} $-\infty=a<0<b<\infty$
\item $-\infty=a<0<b=\infty$
\item $-\infty<a<0=b<\infty$
\item\label{case:C=[a,b]} $-\infty<a<0<b<\infty$
\item $-\infty<a<0<b=\infty$
\item $-\infty<a=0=b<\infty$
\item\label{case:C=[0,b]} $-\infty<a=0<b<\infty$
\item\label{case:C=[0,infty]} $-\infty<a=0<b=\infty$
\end{enumerate}
Direct calculations show that the $C$ from Cases \beqref{case:C=(-infty,0]}, \beqref{case:C=[a,b]},  and \beqref{case:C=[0,infty]} does solve \beqref{eq:FixedPointConvexGeometry}, and $C$ from the other cases does not solve \beqref{eq:FixedPointConvexGeometry}. For instance, consider Case \beqref{case:C=[a,b]}. Since $\gamma<0$, we see that in this case 
\begin{multline*}
(GC)^{\circ}=\{x^*\in\R: x^*\cdot G(c)\leq 1,\,\forall c\in C\}=\{x^*\in\R: x^*\cdot c\geq 1/\gamma,\,\forall c\in [a,b]\}\\
=\{x^*\in\R: x^*\leq 1/(\gamma c):\,\,c\in [a,0)]\}\bigcap \{x^*\in\R: x^*\geq 1/(\gamma c):\,\,c\in (0,b]\}\\
\bigcap\{x^*\in\R: x^*\cdot 0\geq 1/\gamma\}\\
=\left(-\infty,\frac{1}{\gamma a}\right]\cap \left[\frac{1}{\gamma b},\infty\right)\cap \R=\left[\frac{1}{\gamma b},\frac{1}{\gamma a}\right].
\end{multline*}
Thus the equality $C=(GC)^{\circ}$ is possible if and only if $a=1/(\gamma b)$ and $b=1/(\gamma a)$, namely we can take $b>0$ to be arbitrary and then $a=1/(\gamma b)$, as claimed. As another example, consider Case \beqref{case:C=[0,b]}. Then $(GC)^{\circ}=\{x^*\in \R: x^* \cdot\gamma c\leq 1,\,\forall c\in C\}=\{x^*\in \R: x^*\geq 1/(\gamma c): c\in (0,b]\}\cap \{x^*\in\R: x^*\cdot 0\geq 1/\gamma\}=[1/(\gamma b),\infty)\cap \R$. Therefore the equality $C=(GC)^{\circ}$ is possible if and only if $0=1/(\gamma b)$ and $b=\infty$. Since we assumed that $b<\infty$, we see  that no $C$ from Case \beqref{case:C=[0,b]} can solve \beqref{eq:FixedPointConvexGeometry}. The analysis in the  other cases is similar. 
\end{proof}
It is of interest to observe that the solutions $C_{-}=(-\infty,0]$ and $C_+=[0,\infty)$ can be thought of as being  the  limits of the solutions $C_b$ as $b\to 0$ and $b\to\infty$, respectively. We also observe that from Lemma \bref{lem:f=f^*G*} it follows that the function $f_{b}:=\frac{1}{2}\M_{C_{b}}^2$ 
solves the equation $f_{b}(x)=f_{b}^*(-x)$, $x\in X$. This function can be written explicitly as $f_{b}(x):=\frac{b^2}{2}x^2$ if $x\in (-\infty,0]$ and $f_{b}(x):=\frac{1}{2b^2}x^2$ if $x\in [0,\infty)$, and it has been discussed (in a slightly different notation) in \cite[Example 13.2, Equation (13.4)]{IusemReemReich2019jour}. 

\section{$G$ is not positive definite: existence}\label{sec:ExistenceGisNotPositiveDefinite}
In this section we present two general existence results in the case where $G$ is not necessarily positive definite. Before formulating these results we need to recall some terminology. 

Given a real (and necessarily separable) Hilbert space $X\neq \{0\}$ and a linear operator $T:X\to X$, we say that $T$ is a diagonal operator if $T$ is bounded and there exists a countable  orthonormal basis (namely, a countable and complete orthonormal system) $(e_j)_{j\in J}$ in $X$ and a sequence $(t_j)_{j\in J}$ of real numbers (here $J$ can be finite or countable  infinite) such that $T(x_j)_{j\in J}=(t_j x_j)_{j\in J}$ for all $(x_j)_{j\in J}\in X$, where the tuple $(x_j)_{j\in J}$ of real numbers represents the vector $x:=\sum_{j\in J}x_j e_j\in X$. Of course, since $X$ is real, a diagonal operator is self-adjoint. A linear operator $U:X\to X$ is called unitary if $U$ is bounded, invertible and satisfies $U^{-1}=U^*$. We say that a linear operator $G:X\to X$ can be diagonalized to $T$ using a unitary operator $U:X\to X$ (or, in short, that $G$ can be diagonalized) if  $T:X\to X$ is a diagonal operator and we have $T=U^{-1}G U$.  We denote by $\abs(T):X\to X$ the diagonal operator defined by $\abs(T)(x_j)_{j\in J}:=(|t_j|x_j)_{j\in J}$ for each $x=(x_j)_{j\in J}\in X$. It is well known and can be verified without much difficulty that if $G$ can be diagonalized, then $G$ is necessarily bounded and self-adjoint. If we also assume that $G$ is invertible, then a simple verification shows that $T$ is invertible too and, since both $T$ and $T^{-1}$ are bounded, there are positive numbers $\mu_1\leq \mu_2$ such that $\mu_1\leq |t_j|\leq \mu_2$ for each $j\in J$. In particular, both $\abs(T)$ and $U\abs(T)U^{-1}$ are positive definite and invertible. 

\begin{prop}\label{prop:ExistenceGisNotPositiveDefinite}
Given a real Hilbert space $X$, suppose that  $G:X\to X$ is linear, continuous and invertible.  If there exists some $A:X\to X$ which is positive definite and invertible, and also  satisfies  the operator equation 
\begin{equation}\label{eq:AG}
A=GA^{-1}G^*,
\end{equation}
then the ellipsoid $C:=\{x\in X: \langle Ax,x\rangle\leq 1\}$ solves \beqref{eq:FixedPointConvexGeometry}. In particular, if $G$ can be diagonalized to a diagonal operator $T$ using a unitary operator $U:X\to X$, then \beqref{eq:FixedPointConvexGeometry} has a solution which is an ellipsoid: 
this is the ellipsoid $C$ induced by the operator $A:=U\abs(T)U^{-1}$. 
\end{prop}
\begin{proof}
An immediate consequence of \beqref{eq:AG} is the equality 
\begin{equation}\label{eq:A^{-1}=(G^{-1})^*AG^{-1}}
A^{-1}=(G^{-1})^*AG^{-1}.
\end{equation}
Now, since $A$ is positive definite and invertible, and since $G$ is invertible, the change of variables $y:=Gx$,  Lemma \bref{lem:Ellipsoid}\beqref{item:EllipsoidPolarSet} and \beqref{eq:A^{-1}=(G^{-1})^*AG^{-1}} imply that 
\begin{multline}\label{eq:GA-ellipsoid}
(GC)^{\circ}=\{Gx\in X: \langle Ax,x\rangle\leq 1\}^{\circ}
=\{y\in X: \langle AG^{-1}y,G^{-1}y\rangle\leq 1\}^{\circ}\\
=\{y\in X: \langle (G^{-1})^*AG^{-1}y,y\rangle\leq 1\}^{\circ}
=\{y\in X: \langle A^{-1}y,y\rangle\leq 1\}^{\circ}\\
=\{y\in X: \langle Ay,y\rangle\leq 1\}=C,
\end{multline}
namely \beqref{eq:FixedPointConvexGeometry} holds. Finally, suppose that $T=U^{-1}GU$ for some unitary operator $U$ and a diagonal operator $T$, and let $A:=U\abs(T)U^{-1}$. Since $G=UTU^{-1}$, $T=T^*$, $T^2=(\abs(T))^2$, $U^{-1}=U^*$ and since $T$ commutes with other diagonal operators, we have  
\begin{multline*}
GA^{-1}G^*=UTU^{-1}U(\abs(T))^{-1}U^{-1}(U^{-1})^*T^*U^*=UT(\abs(T))^{-1}T^*U^*\\
=UT^2(\abs(T))^{-1}U^{-1}=U\abs(T)U^{-1}=A,
\end{multline*}
that is, \beqref{eq:AG} holds. Since $A$ is positive definite, we can conclude from previous lines that $C:=\{x\in X: \langle Ax,x\rangle\leq 1\}$ solves \beqref{eq:FixedPointConvexGeometry}.
\end{proof}
Interestingly, a similar equation to \beqref{eq:AG} appears in \cite[Lemma 7.1, first equation in (7.1)]{IusemReemReich2019jour} (with $\tau=1$, $G=E^*$) in a different context.

We finish this short section with the following corollary which is an immediate consequence of Proposition \bref{prop:ExistenceGisNotPositiveDefinite} and the well-known result which says that a self-adjoint operator which acts on a finite-dimensional Euclidean space can be diagonalized.
\begin{cor}\label{cor:SelfAdjoint}
If  $X$ is finite-dimensional and $G$ is self-adjoint, then \beqref{eq:FixedPointConvexGeometry} always has at least one solution, and this solution is an ellipsoid.
\end{cor}

\section{$G$ is not positive definite: non-uniqueness}\label{sec:NonUniqueness}
In this section we show, by means of examples and propositions, that when $G$ is not positive definite, then non-uniqueness of the solution to \beqref{eq:FixedPointConvexGeometry} can hold. We also discuss a case (Proposition \bref{prop:BallUniqueEllipsoid}) where uniqueness does hold if one restricts attention to a special subclass of solutions. 

\begin{expl}\label{ex:Cone}
In this example we consider the case where  $G=-I$, namely \beqref{eq:FixedPointConvexGeometry} becomes $C=(-C)^{\circ}$. 

Assume first that $X$ is any real Hilbert space satisfying $\dim(X)\geq 2$ (including the case where $X$ is infinite-dimensional). One solution to \beqref{eq:FixedPointConvexGeometry} is, of course, the unit ball (see also Proposition \bref{prop:ExistenceGisNotPositiveDefinite} above). In order to obtain additional solutions we fix an arbitrary unit vector $x_0\in X$ and define    
\begin{equation}\label{eq:Cone}
C(x_0):=\{0\}\bigcup\left\{0\neq c\in X: \left\langle \frac{c}{\|c\|},x_0\right\rangle\geq \frac{1}{\sqrt{2}}\right\}.
\end{equation}
The set $C(x_0)$ is  a circular cone with its main symmetry axis in the direction of $x_0$ and having half-aperture $\pi/4$, that is, the angle between any $c\in C(x_0)$ and $x_0$ is at most $\pi/4$. It is sometimes called  ``the ice-cream cone'' or ``the Lorentz cone''. The cone $C(x_0)$  solves \beqref{eq:FixedPointConvexGeometry}. This claim in scattered in various forms and settings in the literature, for instance in \cite[p. 140]{Cegielski2012book} and \cite[p. 392]{Goffin1980jour}. A proof of it can be found in \cite[pp. 16--17]{ReemReich2017b-arXiv[v1]}. A closely related example, which is, in fact, the above one in disguise, appears in \cite[p. 130]{BauschkeCombettes2017book} and \cite[p. 28]{DraganMorozanStoica2010book} (note: in the terminology of \cite{BauschkeCombettes2017book,DraganMorozanStoica2010book} the claim that $C(x_0)$ is a self-dual cone implies, in our terminology, that $C(x_0)=(-C(x_0))^{\circ}$). Since $x_0$ can be an arbitrary unit vector, we conclude that \beqref{eq:FixedPointConvexGeometry} has infinitely many solutions which are unbounded and contain the origin on their boundaries. 

Assume now that $X=\R^n$ for some $n\in\N$. Then a simple verification shows that the positive orthant 
\begin{equation*}
C:=\{x=(x_k)_{k=1}^n\in X: x_k\geq 0,\,\forall k\in\{1,\ldots,n\}\} 
\end{equation*} 
solves \beqref{eq:FixedPointConvexGeometry}, as well as any rotation of this orthant (see Proposition \bref{prop:A*-1GA-1=G} below). This claim can be extended to infinite-dimensional Hilbert spaces. As in the case of $C(x_0)$, the claim related to the positive  orthant is also scattered in various forms in the literature: see, for instance, \cite[p. 122]{BauschkeCombettes2017book}, \cite[p. 23]{DraganMorozanStoica2010book} and \cite[p. 392]{Goffin1980jour}. 
\end{expl}

\begin{expl}\label{ex:Simplex}
Here we still consider the case where $G=-I$ and we assume that $X=\R^n$ for some $n\in\N$, $n\geq 2$. We show below that certain regular simplices in $X$ solve \beqref{eq:FixedPointConvexGeometry}. It is not clear to us if this fact has ever been proved formally in either a published or an unpublished form before our paper, but we note that it has been observed before us (for instance, around the beginning of 2017 Dor-On \cite{Dor-On2017pc} observed it in the case $n=3$ and conjectured that it holds for all $n\in\N$). 

Fix $r>0$ and consider a regular simplex $S(r)\subset X$ having circumradius $r$. Basic properties of $S(r)$ (which can actually be used inductively in order to construct $S(r)$) are: 
\begin{itemize}
\item Its center of mass (namely, its centroid) is the origin; 
\item each one of its $n+1$ vertices $v_1,\ldots,v_{n+1}$ is located at distance $r$ from the origin, namely, $\|v_i\|=r$ for each $i\in I:=\{1,\ldots,n+1\}$; 
\item there is a quantitative one-to-one correspondence between its facets and its vertices: on the one hand, given a facet $F$ of $S(r)$, the ray which emanates from the center of mass  of $F$ and passes through the origin also passes via some (unique) vertex of $S(r)$. On the other hand, given a vertex $v$ of $S(r)$, the ray which emanates from $v$ and passes via the origin hits one (and only one) of the facets of $S(r)$ and is orthogonal to this facet; the hitting point is the vector $-v/n$ which also coincides with the center of mass of this facet.  
\end{itemize}

It follows that we can write $S(r)=\cap_{i=1}^{n+1}S_i(r)$, where $S_i(r)$ is the halfspace which contains 0 and also contains, on its boundary, the $i$-th facet of $S(r)$ (that is, $S_i(r):=\{x\in X: \langle x,-v_i/n\rangle\leq r^2/n^2\}$ for all $i\in I$). On the other hand, since $S(r)=\conv\{v_1,\ldots,v_{n+1}\}$, we can use the well-known and elementary result about the polar of a convex hull of finitely many points \cite[p. 144]{Barvinok2002book} to get that $S(r)^{\circ}=\cap_{i=1}^{n+1}S'_i(r)$, where  $S'_i(r):=\{x\in X: \langle x,v_i\rangle\leq 1\}$ for each $i\in I$.

Our goal is to find some $r>0$ such that $S(r)$ solves \beqref{eq:FixedPointConvexGeometry}. From the previous paragraph and the immediate identity $(-C)^{\circ}=-C^{\circ}$ which holds for every nonempty set $C\subseteq X$ (see Lemma \bref{lem:gamma-C-is-one-to-one}\beqref{item:gamma_C composition}), we have $(-S(r))^{\circ}=-S(r)^{\circ}=-\cap_{i=1}^{n+1}S'_i(r)$. 
By recalling the immediate identity $-\cap_{j\in J}A_j=\cap_{j\in J}(-A_j)$ which holds for any collection $(A_j)_{j\in J}$ of nonempty subsets of $X$ (here $J$ is an arbitrary nonempty index set), we see that $-\cap_{i=1}^{n+1}S'_i(r)=\cap_{i=1}^{n+1}(-S'_i(r))$. Since  $S(r)=\cap_{i=1}^{n+1}S_i(r)$, we see that in order to show that $S(r)=(-S(r))^{\circ}$ it is sufficient to find $r>0$ such that $-S'_i(r)=S_i(r)$ for every $i\in I$. Since we can write $S_i(r)=\{x\in X: \langle x,-(n/r^2)v_i\rangle\leq 1\}$ and $-S'_i(r):=\{x\in X: \langle x,-v_i\rangle\leq 1\}$ for each $i\in I$, the choice $r:=\sqrt{n}$ ensures that indeed $-S'_i(r)=S_i(r)$ for every $i\in I$, as required. Of course, any rotation of $S(\sqrt{n})$ solves  \beqref{eq:FixedPointConvexGeometry} as well (see Proposition \bref{prop:A*-1GA-1=G} below).
\end{expl}

\begin{expl}\label{ex:Ellipsoid}
Let $X:=\R^2$ and for each $(x_1,x_2)\in\R^2$ define $G(x_1,x_2):=(x_2,-x_1)$. 
It can be checked directly (or by using Lemma \bref{lem:Ellipsoid}\beqref{item:EllipsoidPolarSet}) that for all $\lambda>0$, the ellipse $C(\lambda):=\{x=(x_1,x_2)\in X: \lambda^2x_1^2+(1/\lambda^2)x_2^2\leq 1\}$ solves \beqref{eq:FixedPointConvexGeometry}.  Similarly, given $n\in \N$, let $X:=\R^{2n}$ and let $G:X\to X$ be defined by $G(x_1,x_2,\ldots,x_{2n-1},x_{2n}):=(x_2,-x_1,\ldots,x_{2n},-x_{2n-1})$ for all $x=(x_i)_{i=1}^{2n}\in X$. Then for all positive numbers $\lambda_1,\ldots,\lambda_n$, the following ellipsoid solves \beqref{eq:FixedPointConvexGeometry}:
\begin{equation*}
C(\lambda_1,\ldots,\lambda_n):=\{x\in X: \sum_{i=1}^n(\lambda_i^2 x_{2i-1}^2+(1/\lambda_i^2)x_{2i}^2)\leq 1\}.
\end{equation*}
\end{expl}

\begin{expl}\label{ex:L1Linfty}
Let $X:=\R^2$ and for each $(x_1,x_2)\in\R^2$ define
\begin{equation*}
G\begin{pmatrix}
x_1\\
x_2
\end{pmatrix}
:=\begin{pmatrix}
   \frac{1}{\sqrt{2}} & \frac{1}{\sqrt{2}} \\
   -\frac{1}{\sqrt{2}} & \frac{1}{\sqrt{2}} 
 \end{pmatrix}
\begin{pmatrix}
x_1\\
x_2
\end{pmatrix}.
\end{equation*}
A direct calculation shows that both the square $C_{\infty}(2^{-0.25}):=\{(x_1,x_2)\in\R^2: \max\{|x_1|,|x_2|\}\leq 2^{-0.25}\}$ and the rhombus $C_1(2^{0.25}):=\{(x_1,x_2)\in\R^2: |x_1|+|x_2|\leq 2^{0.25}\}$ solve \beqref{eq:FixedPointConvexGeometry}. A third solution is the unit disc $B$ since $GB=B$ and $B=B^{\circ}$ (according to Proposition \bref{prop:PositiveDefiniteG}). 
\end{expl}

The next proposition provides further evidence to the non-uniqueness phenomenon. 
\begin{prop}\label{prop:A*-1GA-1=G}
Given a linear, continuous and invertible operator $G:X\to X$, if $C$ solves \beqref{eq:FixedPointConvexGeometry} and there exists a linear, continuous and invertible operator $A:X\to X$ which satisfies
\begin{equation}\label{eq:MoreSolutions}
(A^{-1})^*GA^{-1}=G,
\end{equation}
then $S:=AC$ also solves \beqref{eq:FixedPointConvexGeometry}. In particular, if $G=\lambda I$ for some $0\neq \lambda\in\R$, then given a solution $C$ to \beqref{eq:FixedPointConvexGeometry} and an arbitrary unitary operator $A$, the subset $S:=AC$ also solves \beqref{eq:FixedPointConvexGeometry}, namely any rotation of $C$ also solves \beqref{eq:FixedPointConvexGeometry}. 
\end{prop}

\begin{proof}
Since $C$ solves \beqref{eq:FixedPointConvexGeometry}, we have $AC=A(GC)^{\circ}$. From Lemma \bref{lem:gamma-C-is-one-to-one}\beqref{item:gamma_C composition} and the equality $S=AC$ we have $S=A(GC)^{\circ}=((A^{-1})^*GC)^{\circ}=((A^{-1})^*GA^{-1}S)^{\circ}$. Since we assume that \beqref{eq:MoreSolutions} holds, we conclude that $S=(GS)^{\circ}$, that is, $S$ solves \beqref{eq:FixedPointConvexGeometry}. Finally, if $G=\lambda I$ for some $0\neq\lambda\in\R$, then any unitary operator $A$ satisfies \beqref{eq:MoreSolutions}, and thus from the previous lines we conclude that if $C$ solves the equation $C=(\lambda C)^{\circ}$, then $S:=AC$ also solves this equation. 
\end{proof}
Here are a few comments related to Proposition \bref{prop:A*-1GA-1=G}. 
\begin{remark}\label{rem:GA}
\begin{enumerate}[(i)]
\item\label{item:Ball(lambda)} When $G=\lambda I$ for some $\lambda>0$, then Proposition \bref{prop:PositiveDefiniteG} ensures that the ellipsoid $C=\{x\in X: \langle Gx,x\rangle\leq 1\}$ is the unique solution to \beqref{eq:FixedPointConvexGeometry}. There is no contradiction to Proposition \bref{prop:A*-1GA-1=G} since $C$ is just the ball of radius $1/\sqrt{\lambda}$ about the origin, and it is equal to any rotation of itself. 
\item\label{item:C=AC} The assertion mentioned in Part \beqref{item:Ball(lambda)} can be generalized. More precisely, suppose that \beqref{eq:FixedPointConvexGeometry} has a unique solution $C$. For example, this happens if $G:X\to X$ is positive definite and invertible, as follows from Proposition \bref{prop:PositiveDefiniteG} which ensures that the ellipsoid $C=\{x\in X: \langle Gx,x\rangle\leq 1\}$ is the unique solution to \beqref{eq:FixedPointConvexGeometry}. Assume further that some linear, continuous and  invertible operator $A:X\to X$ satisfies \beqref{eq:MoreSolutions}. Then Proposition \bref{prop:A*-1GA-1=G} implies that $S:=AC$ solves \beqref{eq:FixedPointConvexGeometry} as well. The uniqueness of the solution to \beqref{eq:FixedPointConvexGeometry} implies that $C=AC$. An illustration of this assertion with an operator $G$ which is usually non-scalar is described below: here $X=\R^3$,  and $G$ and $A$ are the linear operators having matrix forms
\begin{equation*}
\wt{G}:=\begin{pmatrix}
   1  & 0 &  0\\
   0 & 1  &  0\\
	0  & 0  &  \lambda
 \end{pmatrix}
\quad \textnormal{and}\quad 
\wt{A}:=\begin{pmatrix}
   \cos(\alpha)  & -\sin(\alpha) &  0\\
   \sin(\alpha) & \cos(\alpha) &  0\\
	0                    & 0                  &  1
 \end{pmatrix},
\end{equation*}
respectively, where $\alpha\in [0,2\pi]$ is arbitrary and $\lambda$ is an arbitrary positive number. Indeed, direct calculations show that \beqref{eq:MoreSolutions} holds and $AC=C$. Alternatively, one can observe that  the ellipsoid is $C=\{(x_1,x_2,x_3)\in X: x_1^2+x_2^2+(\lambda x_3)^2\leq 1\}$ and $A$ rotates this ellipsoid counterclockwise by an angle $\alpha$ about the $x_3$ axis. Since the $x_3$ axis is an axis of symmetry for this ellipsoid, we have $AC=C$. 
\end{enumerate}
\end{remark}

Proposition \bref{prop:X=R} and Examples \bref{ex:Cone}--\bref{ex:Simplex} above show that when $G=-I$, then \beqref{eq:FixedPointConvexGeometry} has many solutions. Proposition \bref{prop:BallUniqueEllipsoid} below  shows that under further assumptions on the class of possible solutions,  uniqueness does hold. 
\begin{prop}\label{prop:BallUniqueEllipsoid}
Suppose that  $G:X\to X$ is defined by $G:=-I$. Then the unit  ball is the unique solution to \beqref{eq:FixedPointConvexGeometry} in the class of all the centrally symmetric ellipsoids of $X$ which are induced by invertible positive definite operators.
\end{prop}
\begin{proof}
Let $B$ be the unit ball. Since $B=GB$ and $B=B^{\circ}$ (from Proposition \bref{prop:PositiveDefiniteG}), it follows that $B$  solves \beqref{eq:FixedPointConvexGeometry}. Suppose now that some centrally symmetric ellipsoid $C$, which is induced by an invertible positive definite operator $A$, also solves \beqref{eq:FixedPointConvexGeometry}. We need to prove that $C=B$. The equalities $C=\{x\in X: \langle Ax,x\rangle\leq 1\}$ and $C=(GC)^{\circ}=(-C)^{\circ}$, and Lemma \bref{lem:Ellipsoid}\beqref{item:EllipsoidPolarSet}, imply that  
\begin{multline}\label{eq:C=D}
C=(GC)^{\circ}=\{-x\in X: \langle Ax,x\rangle\leq 1\}^{\circ}
=\{y\in X: \langle A(-y),-y\rangle\leq 1\}^{\circ}\\
=\{y\in X: \langle Ay,y\rangle\leq 1\}^{\circ}
=\{y\in X: \langle A^{-1}y,y\rangle\leq 1\}=:D.
\end{multline}
Denote $f:=\frac{1}{2}\M_D^2$. Since $D\in\Kbo(X)$ (according to Lemma \bref{lem:Ellipsoid}\beqref{item:EllipsoidIsBounded}), it follows from Lemma \bref{lem:gamma_C_conjugate_polar} and Lemma \bref{lem:(gamma_C)_polar=gamma_(C_polar)} that $f^*=\frac{1}{2}\M_{D^{\circ}}^2$. But $D^{\circ}=\{x\in X: \langle Ax,x\rangle\leq 1\}$ according to Lemma \bref{lem:Ellipsoid}\beqref{item:EllipsoidPolarSet}, namely $D^{\circ}=C$. By recalling that $D=C$ (according to \beqref{eq:C=D}), we get $f^*=\frac{1}{2}\M_{D^{\circ}}^2=\frac{1}{2}\M_C^2=\frac{1}{2}\M_D^2=f$. Thus, by a classical result in convex analysis (see also \cite[Proposition 9.1]{IusemReemReich2019jour} for a more general statement), $f(x)=\frac{1}{2}\|x\|^2$ for each $x\in X$, namely $f=\frac{1}{2}\M_B^2$, that is, $\frac{1}{2}\M_B^2=\frac{1}{2}\M_D^2$. Since the Minkowski functional is nonnegative, it follows that $\M_D=\M_B$. We conclude from Lemma \bref{lem:gamma-C-is-one-to-one}\beqref{item:gamma_C is one-to-one} that $D=B$. The assertion follows since $C=D$. 
\end{proof}

\section{$G$ is not positive definite: non-existence}\label{sec:NonExistence}
In this section we show that if $G$ belongs to a class of ``semi-skew operators'', then \beqref{eq:FixedPointConvexGeometry} does not have any solution $C\in \Kbo(X)$. Semi-skew operators are defined as follows:
\begin{defin}\label{def:SemiSkew}
Let $X$ be a real Hilbert space of dimension at least 2. We say that $E:X\to X$ is a semi-skew operator with respect to the triplet $(u,\alpha_1,\alpha_2)$ (or, briefly, that $E$ is semi-skew) if the following conditions hold:  
\begin{enumerate}[(i)]
\item $u\in X$ is a unit vector;
\item $\alpha_1$ and $\alpha_2$ are two real numbers having the same sign (either both of them are positive or both are negative) and $\alpha_1\neq\alpha_2$;
\item for each $x\in X$, consider the unique decomposition $x=x_1+x_2$, where $x_1\in\R u$ and $x_2\in u^{\bot}$  and identify $x$ with $(x_1,x_2)\in \R u\times u^{\bot}\cong X$ and with $(x_2,x_1)\in u^{\bot}\times \R u\cong X$; then $E(x_1,x_2):=(\alpha_2 x_2,-\alpha_1 x_1)$. In other words, 

\begin{equation*}
E\begin{pmatrix}
x_1\\
x_2
\end{pmatrix}
:=\begin{pmatrix}
   0 & \alpha_2 \\
   -\alpha_1 & 0 
 \end{pmatrix}
\begin{pmatrix}
x_1\\
x_2
\end{pmatrix}=\alpha_2 P_{u^{\bot}}x-\alpha_1P_{\R u}x. 
\end{equation*} 
\end{enumerate}
\end{defin}
The following lemma is an immediate consequence of Definition \bref{def:SemiSkew}.
\begin{lem}\label{lem:SemiSkewProperties}
Suppose that $X$ is a real Hilbert space of dimension at least 2 and that $E:X\to X$ is a semi-skew operator with respect to $(u,\alpha_1,\alpha_2)$. Then $E$ is linear, continuous and invertible. Moreover, for each $(x_1,x_2)\in \R u\times u^{\bot}\cong X$ the following identities hold: $E^*(x_1,x_2)=(-\alpha_1 x_2,\alpha_2 x_1)$, $E^{-1}(x_1,x_2)=(-x_2/\alpha_1,x_1/\alpha_2)$ and $E^{-1}E^*(x_1,x_2)=(-(\alpha_2/\alpha_1)x_1,-(\alpha_1/\alpha_2)x_2)$. In particular, $E^*$ is semi-skew with respect to $(u,-\alpha_2,-\alpha_1)$ and $E^{-1}$ is semi-skew with respect to $(u,-1/\alpha_2,-1/\alpha_1)$.
\end{lem}

The next lemma, which is used in the proof of Lemma \bref{lem:NonExistenceSemiSkew}, is a special case of some results proved in \cite{IusemReemReich2019jour}. 
\begin{lem}\label{lem:FunctionalEquation}
{\bf (A special case of \cite[Lemma 5.1 and Equation (5.2) in Lemma 5.2]{IusemReemReich2019jour}): }
Suppose that $X$ is a real Hilbert space and that $E:X\to X$ is a linear, continuous and invertible operator. If $f:X\to [-\infty,\infty]$ solves the equation 
\begin{equation}\label{eq:SemiSkew}
f(x)=f^*(Ex), \quad x\in X,
\end{equation}
then $f$ is convex, proper and lower semicontinuous. Moreover, it satisfies the following functional equation:
\begin{equation}\label{eq:FunctionalEquation}
f(x)=f(E^{-1}E^*x),\quad  x\in X.
\end{equation}
\end{lem}

Using Lemma \bref{lem:FunctionalEquation}, we are able to prove Lemma \bref{lem:NonExistenceSemiSkew} and using this latter lemma, we can prove Proposition \bref{prop:NonExistence}.
\begin{lem}\label{lem:NonExistenceSemiSkew}
Let $X$ be a real Hilbert space of dimension at least two and suppose that $E:X\to X$ is a semi-skew linear operator with respect to some triplet $(u,\alpha_1,\alpha_2)$. Then there does not exist any solution $f:X\to(-\infty,\infty]$ to \beqref{eq:SemiSkew} which is upper semicontinuous at 0 and satisfies $f(0)\in\R$. 
\end{lem}
\begin{proof}
Suppose to the contrary that some function $f:X\to\R$ which is upper semicontinuous at 0 solves \beqref{eq:SemiSkew} and satisfies $f(0)\in\R$. Lemma \bref{lem:FunctionalEquation} implies that $f$ satisfies \beqref{eq:FunctionalEquation}. Since $E$ is semi-skew with respect to $(u,\alpha_1,\alpha_2)$, we get from Lemma \bref{lem:SemiSkewProperties} and \beqref{eq:FunctionalEquation} that  
\begin{equation}\label{eq:f=Sf}
f(x_1,x_2)=f\left(-\frac{\alpha_2}{\alpha_1}x_1,-\frac{\alpha_1}{\alpha_2}x_2\right),\quad (x_1,x_2)\in \R u\times u^{\bot}\cong X.
\end{equation}
Since $\alpha_1$ and $\alpha_2$ have the same sign and $\alpha_1\neq\alpha_2$, either $0<\alpha_1/\alpha_2<1$ or $0<\alpha_2/\alpha_1<1$. Assume the first case; the proof in the second case is similar (by performing the operations below on the first component instead of on the second one and vice versa). Denote $\alpha:=\alpha_1/\alpha_2$. By putting $x_1=0$ and an arbitrary $x_2\in u^{\bot}$ in \beqref{eq:f=Sf}, we see that $f(0,x_2)=f(0,-\alpha x_2)=\ldots=f(0,(-\alpha)^m x_2)$ for every $m\in\N$. Since $\alpha\in (0,1)$ and $f$ is upper semicontinuous at $0=(0,0)$, we have 
 $f(0,x_2)=\limsup_{m\to\infty}f(0,(-\alpha)^m x_2)\leq f(\lim_{m\to\infty}(0,(-\alpha)^m x_2))=f(0,0)$ for each $x_2\in u^{\bot}$. Now, by making the change of variables $(y_1,y_2):=(-(1/\alpha)x_1,-\alpha x_2)$, we obtain from \beqref{eq:f=Sf} the equation $f(-\alpha y_1,-(1/\alpha)y_2)=f(y_1,y_2)$ for all $(y_1,y_2)\in \R u\times u^{\bot}$. Again, since $\alpha\in (0,1)$,  we can use a similar reasoning as in  previous lines to conclude that $f(y_1,0)\leq f(0,0)$ for all $y_1\in \R u$. Now let $x_1\in \R u$ and $x_2\in u^{\bot}$ be arbitrary. We can write $(x_1,x_2)=\frac{1}{2}(y_1,0)+\frac{1}{2}(0,y_2)$ for $y_1:=2x_1\in\R u$ and $y_2:=2x_2\in u^{\bot}$. Since $f$ is convex (Lemma \bref{lem:FunctionalEquation}), it follows from previous lines that 
\begin{equation}\label{eq:f(0,0)}
f(x_1,x_2)\leq \frac{1}{2}f(y_1,0)+\frac{1}{2}f(0,y_2)\leq f(0,0).
\end{equation} 
Since we assume that $f(0,0)\in\R$, we conclude that $f$ is bounded above by the real constant $f(0,0)$. Since $f$ is convex, we conclude from Lemma \bref{lem:UpperBound} that $f$ itself is equal to some real constant, say $f\equiv\sigma\in\R$. But then \beqref{eq:SemiSkew} implies that $f^*(Ex)=\sigma$ for every $x\in X$ and thus the change of variables $y:=Ex$ and the invertibility of $E$ imply that $f^*(y)=\sigma$ for each $y\in X$. This equality is impossible because if $f$ is equal to a real constant, then $f^*(x^*)=\infty$ for all $x^*\neq 0$ as a simple verification based on \beqref{eq:f^*} shows. Hence \beqref{eq:SemiSkew} cannot have any solution $f$ which is upper semicontinuous at 0 and satisfies $f(0)\in\R$. 
\end{proof}

\begin{prop}\label{prop:NonExistence}
If $\dim(X)\geq 2$ and $G:X\to X$ is semi-skew, then \beqref{eq:FixedPointConvexGeometry} does not have any solution which is bounded and contains 0 in its interior.
\end{prop}
\begin{proof}
Suppose to the contrary that some bounded $C$ which contains 0 in its interior solves \beqref{eq:FixedPointConvexGeometry}. Because of Proposition \bref{prop:ClosedConvex0} it follows that $C\in \Kbo(X)$. Thus  Lemma \bref{lem:f=f^*G*} implies that $f:=\frac{1}{2}\M_C^2$ solves \beqref{eq:SemiSkew} with $E:=G^*$. This  statement contradicts Lemma \bref{lem:NonExistenceSemiSkew} since $E$ is semi-skew (because of Lemma \bref{lem:SemiSkewProperties} and the assumption that $G$ is semi-skew),  $f$ is continuous (as a result of \beqref{eq:M_CIsLipschitz} and the assumption that 0 is in the interior of $C$), and $f(0)=0\in\R$. 
\end{proof}
The assumption that $\alpha_1\neq \alpha_2$ in Proposition \bref{prop:NonExistence} (via Definition \bref{def:SemiSkew}) is essential: indeed, a counterexample is described in Example \bref{ex:Ellipsoid}.  

\section{Proof of Theorem \bref{thm:FixedPointConvexGeometry}}\label{sec:Proof}
\begin{proof}
Proposition \bref{prop:ClosedConvex0} implies Part \beqref{item:ClosedConvex}. Part \beqref{item:PositiveDefinite} is implied by Proposition \bref{prop:PositiveDefiniteG}. Part \beqref{item:NotPositiveDefinite} (existence and non-uniqueness) is a consequence of Proposition \bref{prop:ExistenceGisNotPositiveDefinite}, Corollary \bref{cor:SelfAdjoint}, Propositions \bref{prop:X=R} and \bref{prop:A*-1GA-1=G}, and Examples \bref{ex:Cone}--\bref{ex:L1Linfty}. The non-existence part of Part \beqref{item:NotPositiveDefinite} follows from Proposition ~\bref{prop:NonExistence}. 
\end{proof}

\section{Concluding remarks and open problems}\label{sec:ConcludingRemarks}
We finish the paper with the following remarks and open problems.

\begin{remark}\label{rem:LaxMilgram}
As mentioned just before the formulation of the converse of the Lax-Milgram theorem (that is, before Lemma \bref{lem:EllipticOperator} above), our method of proof can be carried over to a more general setting. Indeed, the same conclusion holds in the case where $X$ is a real Banach space and $A:X\to X^*$ is an invertible positive semidefinite linear operator in the following sense: it is linear, continuous,  it is symmetric, namely $\langle Ax, y\rangle=\langle x, Ay\rangle$ for all $(x,y)\in X^2$, and it satisfies the positive semidefinite inequality, namely $\langle Ax,x\rangle\geq 0$ for all $x\in X$ (where $\langle x^*,x\rangle:=x^*(x)=:\langle x,x^*\rangle$ for all $x^*\in X^*$ and $x\in X$). The exact statement is presented in Lemma \bref{lem:EllipticOperatorGeneraized} in the appendix (Subsection \bref{subsec:LaxMilgramGeneralized} below).  

Moreover, Lemma \bref{lem:EllipticOperatorGeneraized} below can help in showing that a real Banach space $X$ is Hilbertian (namely it is  isomorphic to a real Hilbert space) if and only if there exists an invertible positive semidefinite linear operator $A:X\to X^*$. Indeed, the claim is obvious when $X=\{0\}$, and so from now on we assume that $X\neq \{0\}$. Assume first that $X$ is isomorphic to a real Hilbert space. Then 
there is a real Hilbert space $Z$ and an invertible continuous linear operator $L:X\to Z$. In particular, $Z^*\cong X^*$, where the linear isomorphism $\widetilde{L}:X^*\to Z^*$ satisfies $\widetilde{L}(x^*)=x^*\circ L^{-1}$ for each $x^*\in X^*$. Now define $A:X\to X^*$ by $(Ax)(y):=\langle Lx, Ly\rangle$ for all $x,y\in X$, where $\langle\cdot,\cdot\rangle$ is the inner product in $Z$. Then $A$ is positive semidefinite (in fact, coercive with coercivity coefficient $1/\|L^{-1}\|^{2}$) and the invertibility of $L$, together with the Riesz-Fr\'echet representation theorem, imply that $A$ is invertible. 
On the other hand, if $X$ is a real Banach space and there is some $A:X\to X^*$ which is an invertible positive semidefinite linear operator, then by defining the function $M:X^2\to\R$ by $M(x,y):=\langle Ax, y\rangle$ for every $(x,y)\in X^2$ (where now $\langle z^*,z\rangle:=z^*(z)=:\langle z,z^*\rangle$ for all $z^*\in X^*$ and $z\in X$) we can see that $M$ is a symmetric bilinear form which satisfies $|M(x,y)|\leq\|A\|\|x\|\|y\|$, as follows immediately from the assumptions on $A$. In particular, $\sqrt{M(x,x)}\leq \sqrt{\|A\|}\|x\|$ for each $x\in X$. Lemma \bref{lem:EllipticOperatorGeneraized} implies that $\sqrt{M(x,x)}=\sqrt{\langle Ax,x\rangle}\geq \sqrt{\|A^{-1}\|^{-1}}\|x\|$ for each $x\in X$. We conclude that $M$ is an inner product and the  norm induced by it is equivalent to the original norm of $X$, as required. 
\end{remark}

\begin{remark}\label{rem:ConvexGeometryG-NotPositiveDefinite}
It would be of interest to complete the classification of the set of solutions to \beqref{eq:FixedPointConvexGeometry} in the case where $G$ is not positive definite. At the moment this task seems to be out of reach when $\dim(X)\geq 2$, as can be seen by considering Sections \bref{sec:ExistenceGisNotPositiveDefinite}, \bref{sec:NonUniqueness} and \bref{sec:NonExistence}. In this connection, it  would be of interest to consider the case where $G$ is semi-skew and to prove, or disprove,   that \beqref{eq:FixedPointConvexGeometry} cannot have solutions $\emptyset\neq C\subseteq X$ in this case (the approach that we use in Section \bref{sec:NonExistence}  is heavily based on the assumption that $C\in\Kbo(X)$). Perhaps if one restricts attention to special classes of geometric objects, then one may be able to make some good progress related to this classification (Proposition \bref{prop:ClosedConvex0} ensures that we can restrict our attention to the class of all closed and convex subsets of $X$ which contain 0 not necessarily in their interior, but maybe classes of objects which are more restricted will be easier to deal with). On the other hand, when one considers classes which are different from $\Kbo(X)$, then $\M_C$ can have less pleasant properties: for instance, it can attain the value $\infty$. We believe that in these cases the more general definition of the polar of the Minkowski functional, namely $\M_C^{\circ}(x^*):=\inf\{\mu^*\geq 0: \langle x^*,x\rangle\leq \mu^*\M_C(x),\,\forall x\in X\}$ will be of help. We also believe that the technique of conversion of \beqref{eq:FixedPointConvexGeometry} to \beqref{eq:f_Tf} using Lemma \bref{lem:f=f^*G*} (as done in Proposition \bref{prop:NonExistence}), will be of help (here the inverse operation based on Lemma \bref{lem:gamma-C-is-one-to-one}\beqref{item:gamma_C is one-to-one}, whenever it is possible, seems to be useful too; one can also use Lemma \bref{lem:f=f^*G*} in the opposite direction, namely to find solutions to \beqref{eq:f_Tf} using solutions of \beqref{eq:FixedPointConvexGeometry}).  
\end{remark} 

\begin{remark}
 It would be of interest to consider a more general version of \beqref{eq:FixedPointConvexGeometry}, such as
\begin{equation}\label{eq:TB_1B_2}
C=T((GC+B_1)^{\circ})+B_2,
\end{equation}
where $\emptyset\neq C\subseteq X$ is the unknown, $T$ and $G$ are given invertible and continuous linear operators from $X$ to itself (possibly $T$ is positive definite,  even a positive scalar multiplication of the identity, but not necessarily; in fact, it would be of interest also to consider the case where $T$ and $G$ are nonlinear), and $B_1$ and $B_2$ are two given nonempty subsets of $X$ (possibly singletons, possibly convex, but not necessarily), and the sum is the Minkowski sum $S_1+S_2:=\{s_1+s_2: s_1\in S_1,\,s_2\in S_2\}$. In this case the right-hand side operator $C\mapsto T((GC+B_1)^{\circ})+B_2$ is order reversing when it acts on the class of all nonempty subsets $C\subseteq X$. It might be useful,  depending on the given parameters $T$, $G$, $B_1$ and $B_2$, to restrict attention to subclasses of this class, say to all the bounded and convex subsets of $X$ (after verifying that the right-hand side operator $C\mapsto T((GC+B_1)^{\circ})+B_2$ maps this subclass to itself). In this connection, we note that the closely related equation  $C^{\circ}=C+B_2$ was discussed briefly in \cite[Theorem 4.1]{Molchanov2015jour} under certain assumptions. More precisely, the ambient space $X$  is finite-dimensional and $B_2$ satisfies various conditions, among them that it is convex, closed and contains a  ball having 0 as its center.
\end{remark}

\begin{remark} 
 It would be of interest to classify all the solutions to the operator equations \beqref{eq:AG} and \beqref{eq:MoreSolutions} for general $G$ and $A$. Remark \bref{rem:GA}\beqref{item:C=AC} and also the proof of Proposition \bref{prop:ExistenceGisNotPositiveDefinite}  hint to the possibility that at least in some cases, the corresponding solutions might be described using the eigenvalues of the operator $G$. 
\end{remark}
\begin{remark} 
 It would be of interest to solve \beqref{eq:FixedPointConvexGeometry} while adopting the point of view of \cite[p. 147]{Barvinok2002book}, namely to regard $G$ as the unknown instead of $C$; in other words, to fix some nonempty, closed and convex subset $C\subseteq X$ and to find all the linear, continuous and invertible operators $G:X\to X$ such that  \beqref{eq:FixedPointConvexGeometry} holds (or to prove that no such $G$ exists). 
\end{remark}
\begin{remark}\label{rem:OrderReversingTheory}
The present  paper can be considered  a contribution to the theory of fixed points of order reversing mappings in a geometric context, where the considered fixed points are sets. It is of interest to note that in the last decade or so several other papers have investigated fixed points in the context of (computational)  geometry and  sets, among them \cite{AsanoMatousekTokuyama2007b-jour, AsanoMatousekTokuyama2007jour,  ImaiKawamuraMatousekReemTokuyama2010CGTA, KMT2012, KopeckaReemReich2012jour, MonterdeOngay2014jour, Reem2014jour, Reem2018jour, ReemReich2009jour}. The fixed points considered in these papers are called ``zone diagrams'', ``double zone diagrams'', ``trisectors'', and ``$k$-sectors''. Formally, each such fixed point is a tuple of sets which satisfies a certain geometric condition, namely this tuple solves a certain fixed point equation which is formulated in a geometric setting. As in the present paper, also in the case of   geometric fixed points mentioned above the corresponding operators which induce the fixed point equations are  order reversing (with respect to component-wise inclusion). 
\end{remark}

\section{Appendix: proofs of some claims}\label{sec:Appendix}
In this appendix we provide the proofs of the claims mentioned in Section \bref{sec:Tools} without proof and also some claims mentioned in Remark \bref{rem:LaxMilgram}. 

\subsection{Proofs of the claims mentioned in Section \bref{sec:Tools}}\label{subsec:ProofsSectionTools}
\begin{proof}[Proof of Lemma \bref{lem:gamma-C-is-one-to-one}]
\begin{enumerate}[(a)]
\item Fix $x\in C$. We have $x=1\cdot x\in C$ and thus \beqref{eq:gamma_C} implies that $\M_{C}(x)\leq 1$, namely $x\in C(1)$. Thus $C\subseteq C(1)$. On the other hand, suppose that $x\in C(1)$. Then $\M_{C}(x)\leq 1$. If $x=0$, then obviously $x\in C$. Assume that $x\neq 0$. According to \beqref{eq:gamma_C} and the fact that $\M_{C}(x)$ is finite, for each $k\in \N$, there exists $\mu_k\in [\M_{C}(x),\M_{C}(x)+(1/k))$ and $c_k\in C$ such that $x=\mu_k c_k$. Since $x\neq 0$ and $0<\M_C(x)\leq\mu_k\leq \M_C(x)+1\leq 2$ for each $k\in\N$, the equality $c_k=x/\mu_k$ implies that $c_k$ belongs to the line segment $[0.5x,(1/\M_C(x))x]$. This set is compact and hence the sequence $(c_k)_{k\in\N}$ has a subsequence $(c_{k_j})_{j\in\N}$ which converges, with respect to the norm of $X$, to some point $c$ which belongs to this set. Since $C$ is a closed set and since $c_k\in C$ for all $k\in \N$, we have $c\in C$. However, since $\lim_{k\to\infty}\mu_k=\M_{C}(x)$ and $x=\mu_{k}c_{k}$ for every $k\in \N$, it follows that $x=\lim_{j\to\infty}\mu_{k_j}c_{k_j}=\M_{C}(x)c$. Since $0<\M_{C}(x)\leq 1$, we see that $x$ is a convex combination of $0$ and $c$ and hence, by using the convexity of $C$, we conclude that $x\in C$. Therefore $C(1)\subseteq C$ as well. 
\item This claim was proved during the proof of the previous part (when we saw that $(1/\M_C(x))x=c$ for some $c\in C$). 
\item Suppose that $\M_{C_1}=\M_{C_2}$ for some $C_1,C_2\in \Kbo(X)$. Then $C_1(1)=C_2(1)$ and hence Part \beqref{item:C=C(1)} implies that $C_1=C_2$. 
\item 
To see that $\M_{GC}=\M_C\circ G^{-1}$ holds for an arbitrary nonempty subset $C$ of $X$, let  $x\in X$ be arbitrary. Then, as follows from \beqref{eq:gamma_C}, we have $\M_{GC}(x)=\{\mu\geq 0: x\in \mu GC\}=\{\mu\geq 0: G^{-1}x\in \mu C\}=\M_C(G^{-1}x)$, as required. As for the identity $(GC)^{\circ}=(G^*)^{-1}C^{\circ}$, we observe that from \beqref{eq:C^polar} and the change of variables $y^*=G^*x^*$ we have 
\begin{multline}
(GC)^{\circ}=\{Gc: c\in C\}^{\circ}\\
=\{x^*\in X: \langle x^*,Gc\rangle\leq 1\,\,\forall c\in C\}
=\{x^*\in X: \langle G^*x^*,c\rangle\leq 1\,\,\forall c\in C\}\\
=\{(G^*)^{-1}y^*\in X: y^*\in X,\, \langle y^*,c\rangle\leq 1\,\,\forall c\in C\}=(G^*)^{-1}C^{\circ}.
\end{multline}
Now suppose further that $C\in \Kbo(X)$. The linearity of $G$, and its invertibility and continuity (hence Lipschitz continuity) immediately imply that $GC$ is convex, closed and bounded. If $B$ is any open ball which contains the origin and is contained in $C$, then the assumptions that $G$ is linear, continuous and invertible imply (as a result of the open mapping theorem) that $GB$ is an open set which contains the origin and is contained in $GC$. We conclude from the previous lines that $GC\in \Kbo(X)$. As for the identity $(\M_C\circ G^{-1})^{\circ}=\M_C^{\circ}\circ G^*$, when we combine the equality $\M_{GC}=\M_C\circ G^{-1}$ with \beqref{eq:gamma_C^polar}, the change of variables $y=G^{-1}x$, and the definition of $G^*$, we see that for all $x^*\in X$, 
\begin{multline*}
\M_{GC}^{\circ}(x^*)=\sup_{x\neq 0}\frac{\langle x^*,x\rangle}{\M_{GC}(x)}
=\sup_{x\neq 0}\frac{\langle x^*,x\rangle}{\M_C(G^{-1}x))}\\
=\sup_{y\neq 0}\frac{\langle x^*,Gy\rangle}{\M_C(y)}=\sup_{y\neq 0}\frac{\langle G^*x^*,y\rangle}{\M_C(y)}=\M_C^{\circ}(G^*x^*).
\end{multline*}
\end{enumerate}
\end{proof}

\begin{proof}[Proof of Lemma \bref{lem:(gamma_C)_polar=gamma_(C_polar)}]
The convexity of $C^{\circ}$ follows directly from the linear expression in \beqref{eq:C^polar}, which also implies, in view of the continuity of the inner product, that $C^{\circ}$ is closed. To see that the origin is in the interior of $C$, consider the open ball with radius $1/\|C\|$ about the origin. For each $x^*$ in this ball the Cauchy-Schwarz inequality and the fact that the norms of the points in $C$ are bounded by $\|C\|$ imply that $\langle x^*,c\rangle\leq \|x^*\|\|c\|<(1/\|C\|)\|C\|=1$ for all $c\in C$. Hence $x^*\in C^{\circ}$, namely $ C^{\circ}$ contains the above-mentioned ball. To see that $C^{\circ}$ is bounded, let $0\neq x^*\in C^{\circ}$ be arbitrary. Since $C$ contains an open ball of radius $r_C$ about the origin, for each $\alpha\in (0,1)$, we have $c_{\alpha}:=(\alpha r_C/\|x^*\|)x^*\in C$. Thus  $\langle x^*,c_{\alpha}\rangle\leq 1$. Since $\langle x^*,c_{\alpha}\rangle=\alpha r_C\|x^*\|$, it follows that $\|x^*\|\leq 1/(\alpha r_C)$. This inequality obviously holds for $x^*=0$ as well. Hence $\|C^{\circ}\|\leq 1/(\alpha r_C)$ for each $\alpha\in (0,1)$, and thus $\|C^{\circ}\|\leq 1/r_C$.  We conclude from the previous lines that $C^{\circ}\in \Kbo(X)$.

We now turn to the identity $\M_C^{\circ}=\M_{C^{\circ}}$. It is immediate from \beqref{eq:gamma_C} and \beqref{eq:gamma_C^polar} that $\M_C^{\circ}(0)=0=\M_{C^{\circ}}(0)$. Now fix an arbitrary $0\neq x^*\in X$. We claim that $(1/\M_C^{\circ}(x^*))x^*\in C^{\circ}$. Indeed, let $c\in C$. Obviously $\langle x^*/\M_C^{\circ}(x^*),c\rangle\leq 1$ for $c=0$. If $c\neq 0$, then $\M_C(c)>0$ and $\langle x^*,c\rangle/\M_C(c)\leq \M_C^{\circ}(x^*)$ according to \beqref{eq:gamma_C^polar}. Since $\M_C(c)\leq 1$ according to Lemma \bref{lem:gamma-C-is-one-to-one}\beqref{item:C=C(1)} and since $\M_C^{\circ}(x^*)>0$, we have $\langle x^*/\M_C^{\circ}(x^*),c\rangle\leq \M_C(c)\leq 1$. This is true for each $c\in C$ and hence we conclude from \beqref{eq:C^polar} that $x^*/\M_C^{\circ}(x^*)\in C^{\circ}$. Thus $x^*\in \M_C^{\circ}(x^*) C^{\circ}$ and hence \beqref{eq:gamma_C} (with $C^{\circ}$ instead of $C$) implies that $\M_C^{\circ}(x^*)\geq \M_{C^{\circ}}(x^*)$. 

It remains to show that $\M_C^{\circ}(x^*)\leq \M_{C^{\circ}}(x^*)$. Let $\epsilon>0$ and $0\neq c\in C$ be arbitrary. Since $\M_C(c)$ is finite, it follows from \beqref{eq:gamma_C} that there exists $\mu\in [\M_C(c),\M_C(c)+\epsilon)$ such that $c\in\mu C$. Since $0\neq c$, we have $\mu\geq\M_C(c)>0$ and $c/\mu\in C$. Since $C$ is convex and $0\in C$, the inequality $\mu<\M_C(c)+\epsilon$ implies that $c/(\M_C(c)+\epsilon)\in [0,c/\mu]\subseteq C$. Now denote $\wt{C^{\circ}}(x^*):=\{\mu^*\geq 0: x^*\in \mu^* C^{\circ}\}$. The previous paragraph implies that $\M_C^{\circ}(x^*)\in \wt{C^{\circ}}(x^*)$ and hence $\wt{C^{\circ}}(x^*)\neq \emptyset$. Let $\mu^*\in \wt{C^{\circ}}(x^*)$ be arbitrary. The definition of $\wt{C^{\circ}}(x^*)$ implies that there exists $c^*\in C^{\circ}$ such that $x^*=\mu^*c^*$. Since $c^*\in C^{\circ}$, it follows from \beqref{eq:C^polar} that $\langle c^*,c'\rangle\leq 1$ for each $c'\in C$. In particular, $\langle c^*,c/(\M_C(c)+\epsilon)\rangle\leq 1$. This inequality and the equality $x^*=\mu^*c^*$ imply that  $\langle x^*,c/(\M_C(c)+\epsilon)\rangle\leq \mu^*$. By taking the limit $\epsilon\searrow 0$, we have $\langle x^*,c/\M_C(c)\rangle\leq \mu^*$ for every $\mu^*\in\wt{C^{\circ}}(x^*)$. We conclude that  $\langle x^*,c/\M_C(c)\rangle$ is a lower bound for $\wt{C^{\circ}}(x^*)$ and therefore, according to the definition of $\M_{C^{\circ}}(x^*)$ (namely, \beqref{eq:gamma_C} with $C^{\circ}$ instead of $C$), one has $\M_{C^{\circ}}(x^*)\geq \langle x^*,c/\M_C(c)\rangle$. This inequality holds for every $c\in C\backslash\{0\}$. Now let $x\in X\backslash\{0\}$ be arbitrary. Then $x=\alpha c$ for some $c\in C\backslash\{0\}$ and $\alpha>0$ (indeed, as follows from previous lines, we can take $\alpha:=\M_C(x)+\epsilon$ and $c:=x/\alpha$ for every $\epsilon>0$). Since $\M_C$ is positively homogenous, if follows that $x/\M_C(x)=c/\M_C(c)$. Since we already know that $\M_{C^{\circ}}(x^*)\geq \langle x^*,c/\M_C(c)\rangle$, we have $\M_{C^{\circ}}(x^*)\geq \langle x^*,x/\M_C(x)\rangle$ for each $x\in X\backslash\{0\}$ as well. We conclude from \beqref{eq:gamma_C^polar} that indeed $\M_{C^{\circ}}(x^*)\geq \M_C^{\circ}(x^*)$, as required. 
\end{proof}

\begin{proof}[Proof of Lemma \bref{lem:gamma_C_conjugate_polar}]
The assumption $C\in \Kbo(X)$ implies that both $\M_C(z)$ and $\M^{\circ}_C(z)$ belong to $[0,\infty)$ for all $z\in X$, and hence \beqref{eq:gamma_C_conjugate_polar}--\beqref{eq:0.5MC^2} are well defined. Now let $x\in X$ be arbitrary. It can be represented as $x=\lambda\theta$, where $\lambda\geq 0$ and $\theta\in X$ has the property that $\M_C(\theta)=1$ (indeed, for $x=0$ one can take $\lambda=0$ and $\theta:=y/\M_C(y)$ for some $y\neq 0$, and for $x\neq 0$ one can take $\lambda:=\M_C(x)$ and $\theta:=x/\M_C(x)$). This representation, the assumption  $-\phi(t)=-\infty$ for every $t\in (-\infty,0)$, the fact that $\M_C$ is positively homogenous, and the definition of the convex conjugation and polar operations, all imply that
\begin{multline}%\label{eq:(phi(gamma_C))*(x*) derivation}
(\phi\circ\M_C)^*(x^*)=\sup_{x\in X}[\langle x^*,x\rangle-\phi(\M_C(x))]
=\sup_{\lambda\geq 0}\sup_{\{\theta\in X:\,\, \M_C(\theta)=1\}}[\langle x^*,\lambda\theta\rangle-\phi(\M_C(\lambda\theta))]\\
=\sup_{\lambda\geq 0}\sup_{\{\theta\in X:\,\, \M_C(\theta)=1\}}[\lambda\langle x^*,\theta\rangle-\phi(\lambda(\M_C(\theta))]
=\sup_{\lambda\geq 0}\sup_{\{\theta\in X:\,\, \M_C(\theta)=1\}}[\lambda\langle x^*,\theta\rangle-\phi(\lambda)]\\
=\sup_{\lambda\geq 0}\left[-\phi(\lambda)+\lambda\sup_{\{\theta\in X:\,\, \M_C(\theta)=1\}}[\langle x^*,\theta\rangle]\right]
=\sup_{\lambda\geq 0}[-\phi(\lambda)+\lambda\M_C^{\circ}(x^*)]\\
=\sup_{\lambda\in\R}[\lambda\M_C^{\circ}(x^*)-\phi(\lambda)]=\phi^*(\M_C^{\circ}(x^*)).
\end{multline} 
Suppose now that in addition to the assumption that $\phi(t)=\infty$ for every $t\in (-\infty,0)$ we also assume that $\phi$ is differentiable over $[0,\infty)$, that $\phi(0)=0$, and that $\phi'$ is strictly increasing on $[0,\infty)$ and maps it onto itself. It follows immediately that $(\phi')^{-1}$ exists on $[0,\infty)$ and $\phi'(0)=0$.  The definition of $\phi^*$ and the assumption that $-\phi(t)=-\infty$ for $t\in(-\infty,0)$ imply that $\phi^*(\M_C^{\circ}(x^*))=\sup_{\lambda\geq 0}w(\lambda)$, where $w:[0,\infty)\to\R$ is defined by $w(\lambda):=\lambda\M_C^{\circ}(x^*)-\phi(\lambda)$ for each $\lambda\in [0,\infty)$. Since $C\in \Kbo(X)$ it follows from \beqref{eq:Inequality gamma_C^polar} that $\M_C^{\circ}(x^*)>0$ whenever $x^*\neq 0$. This fact, when combined with  elementary analysis and the fact that $(\phi')^{-1}$ exists on $[0,\infty)$, implies that $w$ attains a unique maximum at $\lambda:=(\phi')^{-1}(\M_C^{\circ}(x^*))$ whenever $x^*\neq 0$. But the same conclusion also holds when $x^*=0$ because the equalities $\phi(0)=\phi'(0)=0$ and the fact that $\phi'$ is increasing imply that $\phi$ is increasing as well, and hence $w(\lambda)=-\phi(\lambda)\leq -\phi(0)=w((\phi')^{-1}(0))$ for all $\lambda\in[0,\infty)$. It follows that $\phi^*(\M_C^{\circ}(x^*))=w((\phi')^{-1}(\M_C^{\circ}(x^*)))$, namely the first equality in \beqref{eq:phi(gamma_C)} holds. 

The second equality in \beqref{eq:phi(gamma_C)} is a consequence of the fundamental theorem of calculus and the  well-known identity $a \psi(a)=\int_0^{a}\psi(t)dt+\int_0^{\psi(a)}\psi^{-1}(t)dt$ which holds for every $a\geq 0$ and every $\psi:[0,\infty)\to [0,\infty)$ which is invertible and strictly increasing (thus continuous) and vanishes at 0 (that is, equality in Young's inequality \cite[Theorem 1]{Witkowski2006jour}, \cite[p. 226]{Young1912jour}; in our case $\psi(t):=(\phi')^{-1}(t)$, $t\in [0,\infty)$). Finally, by taking $\phi(t):=\frac{1}{2}t^2$, $t\geq 0$, we obtain \beqref{eq:0.5MC^2} from \beqref{eq:phi(gamma_C)}.
\end{proof}

\begin{proof}[Proof of Lemma \bref{lem:EllipticOperator}]
Define $h(x):=\frac{1}{2}\langle Ax,x\rangle$ for every $x\in X$. Then $h^*(x^*)=\frac{1}{2}\langle A^{-1}x^*,x^*\rangle$ for every $x^*\in X$, as follows from Lemma \bref{lem:QuadConj}. Since $h$ is proper, it is well known and follows immediately from \beqref{eq:f^*} that $h$ satisfies the Young-Fenchel inequality $h^*(x^*)+h(x)\geq \langle x^*, x\rangle$ for all $x,x^*\in X$. Hence, if we fix $\alpha>0$ and put $x^*:=\alpha x$ in this inequality, then for each $x\in X$, we have 
\begin{equation}\label{eq:alpha-beta}
\alpha\|x\|^2\leq \frac{1}{2}\alpha^2\langle A^{-1}x,x\rangle+\frac{1}{2}\langle Ax,x\rangle\\
\leq \frac{1}{2} \alpha^2\|A^{-1}\|\|x\|^2+\frac{1}{2}\langle Ax,x\rangle,
\end{equation}
where we used the definition of the operator norm and the Cauchy-Schwarz inequality in the second inequality above. We immediately obtain \beqref{eq:EllipticOperator} from \beqref{eq:alpha-beta} by taking $\alpha:=1/\|A^{-1}\|$ (obviously $\|A^{-1}\|\neq 0$ since $X\neq\{0\}$ and $A^{-1}$ is invertible). Finally, if $\beta>0$ is any other coercivity coefficient of the quadratic form $x\mapsto \langle Ax,x\rangle$, then $\beta\|x\|^2\leq \langle Ax,x\rangle\leq \|Ax\|\|x\|$ for every $x\in X$. Since $A$ is invertible, by letting $y:=Ax$ we see that $\beta \|A^{-1}y\|\leq \|y\|$ for all $y\in X$, and therefore, from the definition of  $\|A^{-1}\|$, we have $\|A^{-1}\|=\sup\{\|A^{-1}y\|/\|y\|: 0\neq y\in X\}\leq 1/\beta$. Thus $\beta\leq \|A^{-1}\|^{-1}$, as required. 
\end{proof}

\begin{proof}[Proof of Lemma \bref{lem:Ellipsoid}]
We start by proving Part \beqref{item:EllipsoidIsBounded}. Lemma \bref{lem:EllipticOperator} implies that for  $\beta:=1/\|A^{-1}\|>0$ we have $\beta\|x\|^2\leq\langle Ax,x\rangle$ for all $x\in X$ and, in particular, for each $x\in D$. But $\langle Ax,x\rangle\leq 1$ for each $x\in D$. Thus $\|x\|\leq 1/\sqrt{\beta}$ for every $x\in D$, namely $D$ is bounded.  The continuity of the inner product and of $A$ imply that $D$ is closed, and the linearity of $A$ and of the inner product imply that $D$ is convex. Finally, since $A\neq 0$, an immediate verification based on the definition of $A$ and on the Cauchy-Schwarz inequality shows that the open ball of radius $1/\sqrt{\|A\|}$ about the origin is contained in $D$. It follows from the previous lines that indeed $D\in \Kbo(X)$.

We continue with Part \beqref{item:EllipsoidPolarFunction}. The assertion obviously holds for $x=0$. Now fix $0\neq x\in X$. Let $\mu$ be an arbitrary positive number satisfying $x/\mu\in D$. The definition of $D$ and the assumption that $A$ is positive definite imply that  $0<\langle A(x/\mu),x/\mu\rangle\leq 1$, namely $\mu\geq \sqrt{\langle Ax,x\rangle}$ with equality when $\mu=\sqrt{\langle Ax,x\rangle}>0$. Thus $\inf\{\mu\geq 0: x\in \mu D\}=\sqrt{\langle Ax,x\rangle}$.  Hence \beqref{eq:gamma_C} implies that $\M_D(x)=\sqrt{\langle Ax,x\rangle}$ for every $x\in X$. 

Finally, it remains to prove part \beqref{item:EllipsoidPolarSet}. We already know from Part \beqref{item:EllipsoidIsBounded} that $D\in \Kbo(X)$. This fact allows us to apply Lemma \bref{lem:gamma_C_conjugate_polar} (equation \beqref{eq:0.5MC^2}) to $D$, and by recalling that the Minkowski functional and its polar are non-negative, we arrive at $\M_D^{\circ}=\sqrt{2\left(\frac{1}{2}\M_D^2\right)^*}$. Since $A$ is positive definite and invertible, Lemma \bref{lem:QuadConj} implies that the conjugate of $h(x):=\frac{1}{2}\langle Ax,x\rangle$, $x\in X$,  is $h^*(x^*)=\frac{1}{2}\langle A^{-1}x^*,x^*\rangle$, $x^*\in X$. This fact and Part \beqref{item:EllipsoidPolarFunction} imply that $(\frac{1}{2}\M_D^2)^*(x^*)=h^*(x^*)=\frac{1}{2}\langle A^{-1}x^*,x^*\rangle$ for each $x^*\in X$. Combining all of these equalities with Lemma \bref{lem:(gamma_C)_polar=gamma_(C_polar)}, we conclude that $\M_{D^{\circ}}(x^*)=\M_D^{\circ}(x^*)=\sqrt{2\left(\frac{1}{2}\M_D^2\right)^*(x^*)}=\sqrt{\langle A^{-1}x^*,x^*\rangle}$ for all $x^*\in X$. But Part \beqref{item:EllipsoidPolarFunction} implies that $\M_{\wt{D}}(x^*)=\sqrt{\langle A^{-1}x^*,x^*\rangle}$ for all $x^*\in X$, where $\wt{D}:=\{x^*\in X: \langle A^{-1}x^*,x^*\rangle\leq 1\}$. Hence $\M_{\wt{D}}=\M_{D^{\circ}}$, and from Lemma \bref{lem:gamma-C-is-one-to-one}\beqref{item:gamma_C is one-to-one} it follows that $D^{\circ}=\wt{D}$, as claimed. 
\end{proof}

\begin{proof}[Proof of Lemma \bref{lem:UpperBound}]
The assertion is obvious when $X=\{0\}$ and so from now on we suppose that $X\neq \{0\}$. Suppose to the contrary that $f$ is not constant. Since $f$ is proper (actually finite everywhere), convex, and locally (actually globally) bounded above, it is locally Lipschitz continuous at each point (see \cite[Theorem 5.21, p. 69]{VanTiel1984book}). In particular, it is continuous and hence lower semicontinuous on $X$. Therefore it follows from \cite[p. 89, Lemma 6.12, p. 90, Corollary 6.12]{VanTiel1984book} that $f$ can be represented as the pointwise supremum of a nonempty family of minorant affine functions, namely affine functions the graphs of which lie below (weak inequality) the graph of $f$. If all of these affine functions are constant, then so is $f$ itself, a contradiction to what we supposed. Hence one of the minorant affine functions must be non-constant. Therefore  there exist $0\neq a$ in the dual $X^*$ of $X$ and $\gamma\in\R$ such that $f(x)\geq a(x)+\gamma$ for all $x\in X$. Since $a\neq 0$, there exists $y_0\in X$ such that $a(y_0)\neq 0$. Define $x_0:=y_0$ if  $a(y_0)>0$, and $x_0:=-y_0$ if $a(y_0)<0$. It follows that $a(x_0)>0$. Therefore, in view of the linearity of $a$, for all $t>0$ we have $f(tx_0)\geq ta(x_0)+\gamma\to\infty$ as $t\to\infty$. In particular, $f$ cannot be globally bounded above, a conclusion which contradicts our assumption on $f$. Hence $f$ must be constant, as asserted. 
\end{proof}

\begin{proof}[Proof of Lemma \bref{lem:f=f^*G*}]
Since $C\in \Kbo(X)$, Lemma \bref{lem:gamma-C-is-one-to-one}\beqref{item:gamma_C composition} ensures that $GC\in \Kbo(X)$. Since $C$ solves \beqref{eq:FixedPointConvexGeometry}, the equality $C=(GC)^{\circ}$ and  Lemma \bref{lem:(gamma_C)_polar=gamma_(C_polar)} imply that $\M_C=\M_{(GC)^{\circ}}=\M^{\circ}_{GC}$ and thus $\frac{1}{2}\M_C^2=\frac{1}{2}(\M_{GC}^{\circ})^2$. From Lemma \bref{lem:gamma-C-is-one-to-one}\beqref{item:gamma_C composition} one has $\frac{1}{2}(\M_{GC}^{\circ})^2=\frac{1}{2}(\M_C^{\circ}\circ G^*)^2$. But for all $x\in X$, we have  $\frac{1}{2}(\M_C^{\circ}\circ G^*)^2(x)=\frac{1}{2}(\M_C^{\circ}(G^*(x)))^2=\frac{1}{2}(\M_C^{\circ})^2(G^*x)$ simply because of the definition of the composition of functions. Since $\frac{1}{2}(\M_C^{\circ})^2(G^*x)=(\frac{1}{2}\M_C^2)^*(G^*x)$ as a result of Lemma \bref{lem:gamma_C_conjugate_polar} (equation \beqref{eq:0.5MC^2}), this equality and the ones mentioned in previous lines imply that 
$\frac{1}{2}\M_C^2(x)=(\frac{1}{2}\M_C^2)^*(G^*x)$ for every $x\in X$. In other words, if we denote $f:=\frac{1}{2}\M_C^2$, then $f(x)=f^*(G^*x)$ for all $x\in X$, namely $f$ solves \beqref{eq:f=f^*G*}. 
\end{proof}

\subsection{Proofs of some claims mentioned in Remark \bref{rem:LaxMilgram}}\label{subsec:LaxMilgramGeneralized}

Before presenting the proofs, we recall that given a real Banach space $X$, we set $\langle x^*,x\rangle:=x^*(x)=:\langle x,x^*\rangle$ for all $x^*\in X^*$ and $x\in X$. Given $h:X\to [-\infty,\infty]$, its convex conjugate (Legendre-Fenchel transform) is defined by $h^*(x^*):=\sup\{\langle x^*,x\rangle-h(x): x\in X\}$ for all $x^*\in X^*$ and $x\in X$. Given $A:X\to X^*$ we say that  $A$ is an invertible positive semidefinite linear operator if $A$ is linear, continuous,  symmetric (namely $\langle Ax, y\rangle=\langle x, Ay\rangle$ for all $(x,y)\in X^2$), and it satisfies $\langle Ax,x\rangle\geq 0$ for all $x\in X$. An immediate verification shows that if some $A:X\to X^*$ is symmetric, then so is $A^{-1}$. A well-known fact that we need below is that for all $x\in X$, there exists $j(x)\in X^*$ such that $\|x\|=\|j(x)\|$ and $\langle j(x),x\rangle=\|x\|^2$. The proof of this fact can be found, for example, in \cite[Corollary 1.3, p. 3]{Brezis2011book}, and it is just a simple consequence of the Hahn-Banach theorem; the function $j:X\to X^*$ is called ``the normalized duality mapping''. 

\begin{lem}\label{lem:QuadConjGeneralized}
Let $X$ be a real Banach space and let $A:X\to X^*$ be a positive semidefinite invertible operator. For each $x\in X$, denote $h(x):=\frac{1}{2}\langle Ax, x\rangle$. Then $h^*(x^*)=\frac{1}{2}\langle A^{-1}x^*,x^*\rangle$ for all $x^*\in X^*$.
\end{lem}
\begin{proof}
Fix some $x^*\in X^*$. The definition of $h^*$ implies that $h^*(x^*)=\sup\{k(x): x\in X\}$,  
where $k:X\to\R$ is defined by $k(x):=\langle x^*,x\rangle -h(x)$ for all $x\in X$. 
Since $A$ is invertible and self-adjoint, a simple verification shows that $k(x)=k(x_m)-\frac{1}{2}\langle A(x-x_m),x-x_m\rangle$ for all $x\in X$, where $x_m:=A^{-1}x^*$. This identity and the fact that $A$ is positive semidefinite imply that  $k(x)\leq k(x_m)$ for every $x\in X$, namely $k$ attains a global maximum 
at $x_m$. Thus $h^*(x^*)=k(x_m)=\langle x^*,A^{-1}x^*\rangle-\frac{1}{2}\langle AA^{-1}x^*, A^{-1}x^*\rangle=\frac{1}{2}\langle A^{-1}x^*,x^*\rangle$. 
\end{proof}
\begin{lem}\label{lem:EllipticOperatorGeneraized}
Given a real Banach space $X\neq\{0\}$, if $A:X\to X^*$ is a positive semidefinite and invertible linear operator, then $A$ is coercive (in particular, $A$ is positive definite). As a matter of fact, 
\begin{equation}\label{eq:EllipticOperator}
\langle Ax,x\rangle\geq \|A^{-1}\|^{-1}\|x\|^2,\quad \forall\, x\in X  
\end{equation}
and $\|A^{-1}\|^{-1}$ is the optimal (largest possible) coercivity coefficient.  
\end{lem}
\begin{proof}
Define $h(x):=\frac{1}{2}\langle Ax,x\rangle$ for every $x\in X$. Then $h^*(x^*)=\frac{1}{2}\langle A^{-1}x^*,x^*\rangle$ for every $x^*\in X^*$, as follows from Lemma \bref{lem:QuadConjGeneralized}. Since $h$ is proper, it is well known and follows immediately from the definition of $h^*$ that $h$ satisfies the Young-Fenchel inequality $h^*(x^*)+h(x)\geq \langle x^*, x\rangle$ for all $x\in X$ and $x^*\in X^*$. Now fix $\alpha>0$ and put $x^*:=\alpha j(x)$  for each $x\in X$. Using the property of $j(x)$ mentioned in the beginning of this subsection, and also using the definition of the operator norm  and the symmetry of the pairing $\langle\cdot,\cdot\rangle$, we get $\langle x^*,x\rangle=\alpha\|x\|^2$ and 
\begin{multline}\label{eq:A-1x*x*<=}
\langle A^{-1}x^*,x^*\rangle=\langle x^*,A^{-1}x^*\rangle=\alpha^2\langle j(x),A^{-1}j(x)\rangle\\
\leq \alpha^2\|j(x)\|\|A^{-1}\|\|j(x)\|=\alpha^2\|A^{-1}\|\|x\|^2. 
\end{multline}
Therefore 
\begin{multline}\label{eq:alpha-beta-generalized}
\alpha\|x\|^2=\langle x^*, x\rangle\leq h^*(x^*)+h(x)=\frac{1}{2}\langle A^{-1}x^*,x^*\rangle+\frac{1}{2}\langle Ax,x\rangle\\
\leq \frac{1}{2} \alpha^2\|A^{-1}\|\|x\|^2+\frac{1}{2}\langle Ax,x\rangle.
\end{multline}
We immediately obtain \beqref{eq:EllipticOperator} from \beqref{eq:alpha-beta-generalized} by taking $\alpha:=1/\|A^{-1}\|$ (obviously $\|A^{-1}\|\neq 0$ since $X\neq\{0\}$ and $A^{-1}$ is invertible). Finally, if $\beta>0$ is any other coercivity coefficient of the quadratic form $x\mapsto \langle Ax,x\rangle$, then $\beta\|x\|^2\leq \langle Ax,x\rangle\leq \|Ax\|\|x\|$ for every $x\in X$. Since $A$ is invertible, by taking an arbitrary $y\in X^*$ and letting $x:=A^{-1}y$, we see from the inequality $\beta\|x\|^2\leq \|Ax\|\|x\|$ that $\beta\|A^{-1}y\|^2\leq \|y\|\|A^{-1}y\|$. Thus $\beta \|A^{-1}y\|\leq \|y\|$ and therefore, from the definition of  $\|A^{-1}\|$, we have $\|A^{-1}\|=\sup\{\|A^{-1}y\|/\|y\|: 0\neq y\in X^*\}\leq 1/\beta$. Hence $\beta\leq \|A^{-1}\|^{-1}$, as required. 
\end{proof}

%\newpage

\noindent{\bf Acknowledgments}\vspace{0.1cm}\\
We would like to express our thanks to Radu Bo\c{t}, Andrzej Cegielski, Adam Dor-On, Eliahu Levy, Ben Passer, Orr Shalit and Boaz Slomka for helpful remarks. It is a pleasure for us to thank the referee for several  valuable  comments. The second author was partially supported by the Israel Science Foundation (Grants 389/12 and 820/17), by the Fund for the Promotion of Research at the Technion and by the Technion General Research Fund.

%
%\bibliographystyle{acm}
%\bibliography{biblio}
%
%%%%%%%%%%%%%%%%%%%%%%%%%%%%%%%%%%%%%%%%%%%%%%%%%%%%%%%%%%%%%%%%
%%\begin{comment}
%%\end{comment}
%%%%%%%%%%%%%%%%%%%%%%%%%%%%%%%%%%%%%%%%%%%%%%%%%%%%%%%%%%%%%%%%

\end{document}